\documentclass[a4paper,12pt]{amsart}
\usepackage[T1]{fontenc}
\usepackage[english]{babel}
\usepackage{amsthm, amssymb, amsfonts,amsaddr}
\usepackage{dsfont}
\usepackage{mathtools}
\usepackage{mathrsfs} 
\usepackage[margin=1in]{geometry}
\usepackage{graphicx}
\usepackage{csquotes}

\usepackage{fancyvrb}
\usepackage[
backend=bibtex,
doi=false,
url=false,
maxbibnames=99,
backend=biber,
style=numeric,
firstinits=true,
sortcites=true,
sorting=nty
]{biblatex}
\bibliography{references}
\usepackage{csquotes}
\usepackage{hyperref}
\hypersetup{
    colorlinks,
    linkcolor={red!50!black},
    citecolor={blue!50!black},
    urlcolor={blue!80!black}
}
\usepackage{esint}
\numberwithin{equation}{section}
\theoremstyle{plain}
\newtheorem{theorem}{Theorem}[section]
\newtheorem{proposition}[theorem]{Proposition}
\newtheorem{corollary}[theorem]{Corollary}
\newtheorem{lemma}[theorem]{Lemma}
\theoremstyle{remark}
\newtheorem{remark}[theorem]{Remark}

\theoremstyle{definition}
\newtheorem{definition}[theorem]{Definition}

\DeclarePairedDelimiter{\inn}{\langle}{\rangle}
\DeclarePairedDelimiter{\abs}{\lvert}{\rvert}
\DeclarePairedDelimiter{\norm}{\lVert}{\rVert}

\newcommand{\ess}{\operatorname{ess}}
\newcommand{\spann}{\operatorname{span}}

\newcommand{\diva}{\operatorname{div}}
\newcommand{\Haus}{\mathcal{H}}

\newcommand{\Ha}{\mathbf{H}}

\newcommand{\Ric}{\operatorname{Ric}}

\newcommand{\Cp}{\mathcal{P}}
\newcommand{\Pe}{\mathcal{P}_f}
\newcommand{\G}{\mathcal{G}_f}
\newcommand{\E}{\mathcal{E}}
\newcommand{\Efrak}{\mathfrak{E}}

\newcommand{\C}{\mathbb{C}}
\newcommand{\R}{\mathbb{R}}

\usepackage{xcolor}
\title[Free energy minimizers with radial densities]{Free energy minimizers with radial densities: classification and quantitative Stability}
\author{Shrey Aryan\textsuperscript{1} and Lauro Silini\textsuperscript{2}}
\address{\textsuperscript{1}Massachusetts Institute of Technology (MIT), Cambridge, MA, USA\\ \textsuperscript{2}Institute of Science and Technology Austria (ISTA), Klosterneuburg, Austria}
\email{\textsuperscript{1}shrey183(at)mit.edu, \textsuperscript{2}lauro.silini(at)ist.ac.at}

\begin{document}


\begin{abstract}
We study the isoperimetric problem with a potential energy $g$ in $\mathbb{R}^n$ weighted by a radial density $f$ and analyze the geometric properties of minimizers. Notably, we construct two counterexamples demonstrating that, in contrast to the classical isoperimetric case $g = 0$, the condition $\ln(f)'' + g' \geq 0$ does not generally guarantee the global optimality of centered spheres. However, we demonstrate that centered spheres are globally optimal when both $f$ and $g$ are monotone. Additionally, we strengthen this result by deriving a sharp quantitative stability inequality.
\end{abstract}

\maketitle

\section{Introduction}

In this paper, we explore the geometric properties of liquid drops in equilibrium under the influence of an external potential field. This classical problem has attracted significant attention and has been extensively investigated in the Euclidean setting, see for instance \cite{guido-goldman,figalli2011shape,wente1980symmetry,tamanini1984sphericity,mccann1998equilibrium,herring1951some,gonzalez1980existence,finn1980sessile,finn2012equilibrium,avron1983equilibrium,philippis2015regularity,indrei2024equilibrium,indrei2023minimizingfreeenergy}. 
Our long-term goal is to establish similar results in the general Riemannian setting and as a first step, we begin by studying this problem by weighting the Lebesgue measure on $\R^n$ with a density $f=e^\psi$, where $\psi$ is radially symmetric with respect to the origin. This allows us to obtain an elementary model of a Riemannian manifold in normal coordinates. 

More precisely, we are interested in minimizers of the following functional
\begin{equation}\label{eq:def_energy}
    \E(F):=\Pe(F)+\G(F):=\int_{\partial^*F}f\,d\Haus^{n-1}+\int_{F}gf\,dx,
\end{equation}
where $d\Haus^{n-1}$ is the $(n-1)$-dimensional Hausdorff measure, $-\nabla g$ is the potential field (for instance gravity), $dx$ is the Lebesgue measure on $\R^n$, and $F$ is a set of finite perimeter with reduced boundary $\partial^*F$ (we refer to \cite{Maggi} for more details). The operator $\Pe$ represents a non-autonomous isotropic interface surface energy, and $\G$ denotes the bulk potential resulting from the external field. In particular, liquid drops at equilibrium solve
\begin{equation}\label{eq:optimization_problem}\tag{$\ast$}
    \inf\Bigl\{\E(F)=\Pe(F)+\G(F):\abs{F}_f=v\Bigr\},\qquad \abs{F}_f:=\int_{F}f\,dx, 
\end{equation}
where $v>0$ is fixed, and the infimum is taken over the class of finite perimeter sets sharing the same weighted volume. To respect the symmetry of the underlying weighted space, we assume that $g$ is also radially symmetric about the origin. When examining the geometric properties of minimizers, two central questions naturally arise:
\begin{itemize}
    \item[--] \emph{Under which conditions on $f$ and $g$ can one explicitly describe minimizers of \eqref{eq:optimization_problem}?}
    \item[--] \emph{Does proximity to the minimum imply proximity to a minimizer?}
\end{itemize}
While minimizers exist under mild assumptions, the first question is largely unresolved due to the delicate interaction between $\Pe$ and $\G$. To address this we first completely classify a rich class of radial and isotropic energies for which spheres centered at the origin, $\partial B_R$, uniquely solve \eqref{eq:optimization_problem}. A necessary condition on $f=e^\psi$ and $g$ can be derived by analyzing the first and second variation of the energy $\E$:
\begin{equation*}\label{eq:stability_cond_intro}
\text{$\partial B_R$ is stationary and stable }\iff \psi''(R)+g'(R)\geq 0,
\end{equation*}
where \emph{stationary and stable} means that $\partial B_R$ locally minimizes $\E$ up to the second order\footnote{Note that the crucial condition is to be stable since, as we will see in Section \ref{sec:existence_variations}, all centered spheres are automatically stationary if $f,g$ are radial.}. When $g=0$, this condition reduces to the log-convexity of $f$, that is $\psi''(R)\geq 0.$ Brakke famously conjectured that when the dimension $n\geq 2$, then the log-convexity of $f$ is not only necessary but also sufficient for $\partial B_R$ to be minimizers of the weighted isoperimetric problem. This conjecture stimulated considerable work in the last two decades \cite{bayle2003proprietes,rosales2008isoperimetric,morgan-pratelli2013, kolesnikov2011isoperimetric,mcgillivray2018isoperimetric,li2023class,figalli2013isoperimetric}, culminating in its resolution by Chambers in \cite{chambers}. Motivated by this, we are led to consider the following question:
\begin{itemize}
    \item [Q1)]Is $\psi''+g'>0$ also a sufficient condition for $\partial B_R$ to be \emph{global and unique} minimizers of $\E$?
\end{itemize}
For $n\geq 2$ we consider the strict inequality as opposed to $\psi''+g'\geq 0$ in order to avoid non-uniqueness as explained further in Section \ref{subsec:counterexamples} (also about this see \cite[Theorem 1.2]{chambers}). We demonstrate that this holds under the additional assumption that 
$f$ and $g$ are monotone, but it fails in general, as illustrated by two counterexamples in Section \ref{subsec:counterexamples}. These results provide a complete answer to $Q1).$

Building on the existence and uniqueness of minimizers, we turn to the question of almost-rigidity. Specifically, we ask
\begin{itemize}
    \item [Q2)]Under which conditions on $\psi$ and $g$ does the energy gap $\E(E)-\E(B_R)$ provide a control on the $L^1$-distance of $E$ to $B_R$ quantitatively?
\end{itemize}
This question, rooted in the seminal work of Fuglede \cite{fuglede1989stability}, has a rich history in the study of isoperimetric problems. Rigorous results addressing similar questions have been established in various contexts, including the Euclidean setting \cite{fusco2008sharp,cicalese2012selection,fusco2014strong,fusco2015quantitative,cicalese2013best,hall1992quantitative}, Riemannian manifolds \cite{bogelein2015sharp,bogelein2017quantitative,silini2023quantitative,chodosh2022riemannian,figalli2013sharp,carron1996stabilite}, anisotropic energies \cite{figalli2010mass,palmer1998stability,esposito2005quantitative,neumayer2016strong}, and weighted settings \cite{fusco-manna,cinti2022sharp,cianchi2011isoperimetric,indreicones}.

Lastly, we note that the weighted isoperimetric problem has also been studied in curved spaces \cite{morgan2005manifolds,bakry2006diffusions,bakry1996levy,corwin2006differential,gromov1986isoperimetric,bongiovanni2018isoperimetry,howe2015log,mcgillivray2017weighted,scheuer2019locally,silini2024approaching}, highlighting its broad applicability. Furthermore, when $g\neq 0$, weighted isoperimetric problems, featuring distinct weights for perimeter and volume, have been instrumental in the recent resolution of the Stable Bernstein Conjecture for minimal surfaces (see for instance \cite{antonelli2024new,mazet2024stable} for further details).

\subsection{Main results}
Let $r=\abs{x}$ denote the distance coordinate from the origin. For the sake of readability, we will identify radial functions with their profile, meaning that if $h:\R^n\to\R$ is radial, then $h(x)$ is identified with $h(r)$, $\frac{\partial h}{\partial r}$ with $h'(r)$, $\frac{\partial^2 h}{\partial r^2}$ with $h''(r)$, and so on. This applies also in the one dimensional case, in which $f$ and $g$ have to be considered even. We will always denote the Euclidean ball of radius $R$ centered at $x_0$ with $B_R(x_0)$, and we denote with $B_R$ the ball centered at the origin $B_R(0)$.

It will be convenient to introduce the following terminology concerning the conditions we will impose on $f=e^\psi$ and $g$. 
\begin{definition}\label{def:admissible_weights}
    We say that a pair functions $\psi,g:\R\to [0,+\infty)$  are \emph{admissible weights} if the following conditions hold true:
    \begin{enumerate}
        \item $\psi\in C^2(\R)$, $g\in C^1(\R)$, and $\lim_{r\to+\infty} (\psi(r)+g(r))=+\infty$,
        \item $\psi'(0)=0$,
        \item $\psi''(r)+g'(r)\geq 0$, for all $r\geq 0$.
    \end{enumerate}
    We say $\psi,g$ are \emph{strictly admissible weights} if they are admissible and
    \begin{enumerate}   
        \item[(3')] $\psi''(r)+g'(r)>0$, for all $r>0$.
    \end{enumerate}
    We say $\psi,g$ are $\kappa$-\emph{uniformly admissible weights} if they are admissible and if there exists a constant $\kappa>0$ such that
    \begin{enumerate}
        \item[(3'')] $\psi''(r)+g'(r)\geq\kappa>0$, for all $r\geq 0$.
    \end{enumerate}
    
\end{definition}
As we will see in Section \ref{sec:existence_variations}, conditions (1) and (2) ensure the existence of minimizers together with the regularity needed to make sense of conditions (3), (3'), and (3''), which are all variations of increasing strength of the stability requirement for centered spheres discussed before.

Our first result completely addresses Q1) in the one-dimensional setting under the assumptions summarized in Definition \ref{def:admissible_weights}.

\begin{theorem}\label{thm:one_dim}
    Let $n=1$ and $\psi,g$ be radial admissible weights. Then, centered intervals minimize \eqref{eq:optimization_problem} for any volume if and only if $\min_{r\geq 0}\psi(r)=\psi(0)$. Furthermore, if $\psi,g$ are \emph{strictly} admissible and $\min_{r\geq 0}\psi(r)=\psi(0)$, then centered intervals \emph{uniquely} solve \eqref{eq:optimization_problem}.
\end{theorem}

In dimension one and when $\psi\equiv 0$, Indrei 
showed in a recent work \cite{indrei2025one} an analogous result only requiring $g$ to have convex sub-level sets, employing among others tools coming from the theory of optimal transportation. Note that Theorem \ref{thm:one_dim} is sharp, in the sense that uniqueness can fail when $\psi,g$ are merely admissible weights, as the elementary example $\psi\equiv 0$ and $g(x)=0$ when $\abs{x}\leq 1$, $g(x)=(x-1)^2$ when $\abs{x}>1$, clarifies. This result shows that even in dimension one the condition $\psi''+g'\geq 0$ has to be coupled with an extra assumption in order to answer Q1). This result does not contradict the isoperimetric strict log-convex situation ($g\equiv 0$, $\psi''>0$), since in that case, $\psi$ has automatically a minimum at the origin.

Via a calibration technique introduced by Kolesnikov and Zhdanov in \cite{kolesnikov2011isoperimetric}, we show that Q2) always holds in the large volume regime assuming 3'').
\begin{theorem}\label{thm:large_vol}Let $\psi,g$ be $\kappa$-uniform admissible weights. Then, in $\R^n$, $n\geq 2$, centered spheres with radius greater that $\sqrt{\frac{n+2}{\kappa}}$ uniquely minimize \eqref{eq:optimization_problem}.
\end{theorem}
Note that the above result, when $\psi \equiv 0$, was established in \cite[Theorem 2.1]{indrei2023minimizingfreeenergy}. The next proposition shows that, in general, Q1) does not hold even in the more restrictive class of $\kappa$-admissible weights for which $\psi$ is a minimum at the origin in contrast with the one-dimensional case of Theorem \ref{thm:one_dim}.

\begin{proposition}\label{prop:counter_examples}
    In $\R^n$, $n\geq 2$, there exist $\kappa$-uniformly admissible weights $\psi,g$ such that $\min_{r\geq 0}\psi(r)=\psi(0)$, but centered spheres do not always minimize \eqref{eq:optimization_problem}. Moreover, one can construct $\psi$ and $g$ so that either one of them is strictly monotone increasing.
\end{proposition}

The construction of the counterexamples suggests that the monotonicity of the two densities is needed as an additional assumption to approach Q1). In fact, we are able to show the following classification argument.

\begin{theorem}\label{thm:uniqueness}
    Let $\psi,g$  be monotone increasing strictly admissible weights of class $C^3(\R)$. Then, in $\R^n$, the centered spheres uniquely minimize \eqref{eq:optimization_problem} for all volumes.
\end{theorem}
As a direct corollary, we obtain an isoperimetric-type inequality involving the energy and the volume of any finite perimeter set.
\begin{corollary}\label{cor:isop_ineq}
    Let $\psi,g$  be monotone increasing strictly admissible weights of class $C^3(\R)$, $\Phi(r):=\abs{B_r}_f$, and $\tilde F(r):=\E(B_r)$. Then, the function $\Efrak:=\tilde F\circ\Phi^{-1}$ is well defined, satisfies
    \begin{equation}\label{eq:isop_profile}
    \Efrak'(\Phi(r))=g(r)+\psi'(r)+\frac{n-1}{r},\qquad \Efrak''(\Phi(r))=\frac{r^2(g'(r)+\psi''(r))-n+1}{n\omega_nr^{n+1}e^{\psi(r)}},
    \end{equation}
    and 
    \begin{equation}\label{eq:isoperimetric_ineq}
        \E(F)\geq \Efrak(\abs{F}_f),
    \end{equation}
    for all measurable sets $F\subset \R^n$, with equality if and only if $E$ is a centered ball up to a negligible set.
\end{corollary}
\begin{remark}
    Observe that equation \eqref{eq:isop_profile} combined with the admissibility condition of Definition \ref{def:admissible_weights} implies that the energy profile
    \[
    \Efrak(v)=\min\{\E(F):\abs{F}_f=v\}=\tilde F(\Phi^{-1}(v)),
    \]
    is strictly concave when $v$ is small enough. In contrast with the unweighted case, this may fail in the large volume regime. It suffices to consider for example $\kappa$-uniformly admissible weights, in which $\Efrak$ becomes strictly convex for volumes larger than $\Phi\Bigl(\sqrt{\kappa^{-1}(n-1)}\Bigr)$. See also \cite[Corollary 2.6]{indrei2023minimizingfreeenergy} in the case when $\psi\equiv 0.$
\end{remark}

Finally, we succeeded in answering Q2) in the context of Theorem \ref{thm:uniqueness}, showing the following quantitative stability result.
\begin{theorem}\label{thm:stability}Let $\psi,g$  be monotone increasing strictly admissible weights of class $C^3(\R)$, and $R>0$. Then, there exists a constant $c=c(R,\psi,g,n)>0$ such that for all measurable set $E$  with $\abs{E}_f=\abs{B_R}$,
    \begin{equation}
        \E(E)-\E(B_R)\geq c\abs{E\triangle B_R}_f^2.
    \end{equation}
\end{theorem}
Note that the quadratic exponent in the above stability estimate is sharp, which for instance can be seen by evaluating the energy functional on ellipsoids \cite{fusco2015quantitative}. The same result was been established when $\psi\equiv 0$ in \cite[Theorem 2.1]{indrei2023minimizingfreeenergy}.
\subsection{Structure of the paper}
We now provide a brief overview of the paper's structure. Section \ref{sec:existence_variations} begins by addressing the existence, regularity, and boundedness of minimizers as well as the monotonicity of the isoperimetric profile. In Section \ref{sec:one-d}, we focus on the one-dimensional problem and prove Theorem \ref{thm:one_dim}. Section \ref{sec:large-volume} is dedicated to proving Theorem \ref{thm:large_vol}, followed by Section \ref{subsec:counterexamples} where we construct two counterexamples that prove Proposition \ref{prop:counter_examples}. In Section \ref{sec:optimal-sets} we characterize the minimizers of \eqref{eq:optimization_problem} and prove Theorem \ref{thm:uniqueness}. Finally, in Section \ref{sec:stability} we conclude with quantitative stability results including a proof of Theorem \ref{thm:stability}.

\subsection*{Acknowledgments}
The authors express their gratitude to Emanuel Indrei for his insightful discussions and valuable feedback.

\section{Necessary and sufficient conditions on the weights}\label{sec:nec-suf-weights}
\subsection{Existence, regularity, first and second variations, and optimal profile}\label{sec:existence_variations}
We start by stating the result concerning the existence of sets minimizing \eqref{eq:optimization_problem}. The assumptions on the weights at this level are pretty mild. The proof, which follows the lines of Theorem 3.3 in \cite{morgan-pratelli2013}, is based on the observation that if a minimizing sequence has volume far from the origin (that is, escaping to infinity), then it must also have an arbitrarily large radial and tangential perimeter, contradicting optimality in the limit.
\begin{theorem}\label{thm:existence}
    Let $f,g:\R^n\to [0,+\infty)$ radial functions such that $f$ is lower semi-continuous and $g$ is a locally bounded Borel function. Then, as long as $f$ or $g$ diverge to infinity, for every volume $v>0$ there exists a set of finite perimeter $F\subset \R^n$ satisfying $\abs{F}_f=v$ and minimizing \eqref{eq:optimization_problem}.
\end{theorem}
\begin{proof}
We start with existence. Fix any volume $v>0$, and let $(E_k)$ be a minimizing sequence on finite perimeter sets with weighted volume $v$ such that $\lim_{k\to+\infty}\E(E_k)=\inf\{\E(F):\abs{F}_f=v\}$. 
By classical compactness argument \cite{morgan2016geometric,Maggi} splitting the sequence $(E_k)$ into a bounded and a diverging part, one can prove that $E_k$ converges up to a subsequence to a set $E_\infty$ of volume $\abs{E_\infty}_f=v$ provided
\[
\lim_{R\to+\infty}\liminf_{k\to+\infty}\abs{E_k\setminus B_R}=0,
\]
that is, we do not lose volume at infinity. By the lower-semi-continuity, we conclude that $E$ is optimal. Suppose by contradiction that there exists $\varepsilon>0$ such that for every $R>0$ there exists $k(R)\geq 0$ realizing (up to relabeling)
\begin{equation}
    \abs{E_k\setminus B_R}_f\geq\varepsilon,\quad\text{for all } k>k(R).
\end{equation}
We deduce by the coarea formula \cite{Maggi,vol1967spaces,ambrosio2000functions} and the isoperimetric inequality on the sphere that
\begin{align*}
    \Pe(E_k)&\geq\int_R^{+\infty}\Haus^{n-2}(\partial E_k\cap\partial B_r)f(r)\,dr\\
    &\geq c_n\int_R^{+\infty}\Haus^{n-1}(E_k\cap\partial B_r)^{1-\frac{1}{n-1}}f(r)\,dr\\
    &\geq c_n\varepsilon\Bigl(\max_{r\geq R}\Haus^{n-1}(E_k\cap\partial B_r)\Bigr)^{-\frac1{n-1}},
\end{align*}
for some dimensional constant $c_n>0$. As a consequence, we have that
\begin{align*}
\Pe(E_k)^{\frac n{n-1}}&=\Pe(E_k)\Pe(E_k)^{\frac1{n-1}}\geq \Pe(E_k)\inf_{r\geq R}f(r)^{\frac 1{n-1}}\Bigl(\max_{r\geq R}\Haus^{n-1}(E_k\cap \partial B_r)\Bigr)^{\frac1{n-1}}\\
&\geq c_n\varepsilon\inf_{r\geq R}f(r)^{\frac 1{n-1}}.
\end{align*}
We finally deduce that
\begin{align*}
    \E(E_k)\geq \varepsilon\inf_{r\geq R} g(r)+c_n^{\frac{n-1}{n}}\varepsilon^{\frac{n-1}{n}}\inf_{r\geq R}f(r)^{\frac 1n},
\end{align*}
which leads to a contradiction if the right-hand side can be made arbitrarily large independently of $\varepsilon>0$. This is the case if either $f$ or $g$ diverge to infinity.
\end{proof}
The next proposition shows that under quite general growth hypothesis on the weights, the larger the volume, the larger the energy. The proof, which builds on Theorem 4.3 \cite{morgan-pratelli2013}, takes advantage of an additional integration by parts peculiar to the structure $\E$ in order to balance $\Pe$ with $\G$ when estimating the energy gap.
\begin{proposition}\label{prop:increasing_profile} For $n\geq 2$, let $\psi\in C^2$ and $g\in C^1$ be two radial, positive weights such that
\begin{equation}
    \int_0^r \psi''(\tau)+g'(\tau)\,d\tau>\frac{n-1}{r},\quad \forall r>0,
\end{equation}
and $\psi'(0)=0$. Then, the energy profile associated to Problem \eqref{eq:optimization_problem} with weights $f=e^\psi$ and $g$ defined as
    \begin{equation}\label{eq:controlled_osc}
        \mathfrak{E}:v\mapsto\inf\Bigl\{\E(F):\abs{F}_f=v\Bigr\},
    \end{equation}
    satisfies for every fixed volumes $0<v'<v$
    \begin{equation}
    \frac{\mathfrak{E}(v)-\mathfrak{E}(v')}{v-v'}\geq g(0).
    \end{equation}
    Furthermore, the above inequality is strict if a solution to \eqref{eq:optimization_problem} exists for the volume $v$.
\end{proposition}
\begin{remark}
   Notice that condition \eqref{eq:controlled_osc} is satisfied if $\psi$, $g$ are admissible in the sense of Definition \ref{def:admissible_weights}, and therefore combining this result with Theorem \ref{thm:existence} we deduce that $\mathfrak E$ is strictly monotonically increasing in this case. It is however interesting to notice that on the energy level, one can also consider weights for which $\psi''+g'<0$. 
\end{remark}
\begin{proof}
    It is convenient to adopt an anisotropic point of view of Problem \eqref{eq:optimization_problem} in order to estimate the growth of $\mathfrak E$. Let $E\subset\R^n$ be any bounded set of finite perimeter. Integrating by parts we can express the following integral as follows
    \begin{align*}
        \int_{E}f'\,dx&=\int_E \nabla f\cdot\nabla r\,dx=\int_{\partial ^*E}f\partial_\nu r\,d\Haus^{n-1}-\int_E f\Delta r\,d\Haus^{n-1}\\
        &=\int_{\partial ^*E}f\frac{x\cdot\nu}{r}\,d\Haus^{n-1}-\int_E (n-1)r^{-1}f\,dx.
    \end{align*}
    Notice that the above integral converges since $n\geq 2$. We can express
    \[
    g(r)=g(0)-\psi'(r)+K(r),
    \]
    where $K(r)=\int_0^r \psi''+g'\,d\tau$. This implies
    \begin{equation}\label{eq:anisotropic_view}
    \begin{split}
        \E(E)&=\int_{\partial^*E}f\,d\Haus^{n-1}+\int_{E}gf\,dx\\
        &=\int_{\partial^*E}f\,d\Haus^{n-1}+\int_{E}(g(0)-\psi'+K)f\,dx\\
        &=\int_{\partial^*E}\Bigl(1-\frac{x\cdot\nu}{r}\Bigr)f\,d\Haus^{n-1}+\int_{E}(K+g(0)+(n-1)r^{-1})f\,dx.
    \end{split}
    \end{equation}
    Fix now two volumes $0<v'<v$ and a small constant $\delta>0$, and let $E$ be $\delta$-close to be a minimizer of \eqref{eq:optimization_problem} for the volume $v$, meaning that $\E(E)\leq \mathfrak E(v)+\delta$. Let $R>0$ be such that $\abs{E\cap B_R}_f=v'$. Then, on one side
    \[
    \mathfrak E(v')\leq\E(E\cap B_R),
    \]
    and on the other thanks to Equation \eqref{eq:anisotropic_view}, we can estimate
    \begin{align*}
        \E(E)-\E(E\cap B_R)&-\int_{E\setminus B_R}(K+g(0)+(n-1)r^{-1})f\,dx\\
        &=\int_{\partial^*E\setminus B_R}\Bigl(1-\frac{x\cdot\nu}{r}\Bigr)f\,d\Haus^{n-1}-\int_{E\cap \partial B_R}\Bigl(1-\frac{x\cdot\nu}{r}\Bigr)f\,d\Haus^{n-1}\\
        &=\int_{\partial^*E\setminus B_R}\Bigl(1-\frac{x\cdot\nu}{r}\Bigr)f\,d\Haus^{n-1}\geq 0.
    \end{align*}
    We deduce, thanks to condition \eqref{eq:controlled_osc} ensuring $K(r)-(n-1)r^{-1}>0$, that
    \begin{equation}
        \mathfrak E(v)-\mathfrak E(v')\geq \E(E)-\E(E\cap B_R)-\delta> g(0)(v-v')-\delta,
    \end{equation}
    concluding the proof by arbitrariness of $\delta>0$. Finally, if there exists $E$ globally minimizing \eqref{eq:optimization_problem}, the above argument works with $\delta=0$, proving the strict inequality in this case.
\end{proof}
As a consequence, we obtain boundedness of minimizers.
\begin{theorem}\label{thm:boundedness}
    Let $\psi,g$ be admissible weights in Definition \ref{def:admissible_weights}. Then, sets minimizing \eqref{eq:optimization_problem} are bounded.
\end{theorem}
\begin{proof}
Let $E$ be an optimal set for a given arbitrary volume. In the following, we will denote with $\Haus^d_f:=f\Haus^d$ the $d$-dimensional Hausdorff measure weighted by $f$. For $r>0$ define the following functions:
\[
    P(r):=\Haus_f^{n-1}(\partial E\setminus B_r),\quad V(r):=\abs{E\setminus B_r}_f,\quad G(r):=\G(E\setminus B_r).
\]
Notice that for almost every $r>0$, we have that
\[
P'(r)=-\Haus^{n-2}_f(\partial E\cap\partial B_r),\quad V'(r)=-\Haus_f^{n-1}(E\cap\partial B_r),\quad G'(r)=g(r)V'(r),
\]
and
\[
\lim_{r\to+\infty}\bigl( G(r)+V(r)\bigr)=0,
\]
since $\abs{E}_f+\G(E)<\infty$. By Proposition \ref{prop:increasing_profile}, we have that $\E(E)\geq\E(E\cap B_r)$, so in particular
\[
\Pe(E)+\G(E)\geq\G(E)-G(r)+\Pe(E)-P(r)+\Haus^{d-1}_f(E\cap\partial B_r),
\]
implying that
\begin{equation}\label{eq:estimate_growth_GP}
    G(r)+P(r)\geq -V'(r).
\end{equation}
The isoperimetric inequality on the sphere, already used in Theorem \ref{thm:existence}, implies (for $r>0$ big enough) that
\[
\Haus^{n-2}_f(\partial E\cap\partial B_r)\geq c_n\Haus_f^{n-1}(E\cap\partial B_r)^{\frac{n-2}{n-1}},
\]
which combined with \eqref{eq:estimate_growth_GP} and $G'(r)=g(r)V'(r)\leq 0$ gives
\begin{equation}
    -(P'(r)+G'(r))\geq -c_nV'(r)\Bigl(c_n(G(r)+P(r))\Bigr)^{-\frac 1{n-1}}.
\end{equation}
Integrating from $r$ to infinity, we get the estimate
\begin{equation}\label{eq:estimate_growth_GP2}
    (P(r)+G(r))^{\frac n{n-1}}\geq c_nV(r).
\end{equation}
Let $\varepsilon>0$ be small enough, and set $r_0>0$ so that $V(r)<\varepsilon$ for all $r>r_0$. Performing a volume compensation argument Lemma 17.21 \cite{Maggi}, one can perturb $E$ inside $B_{r_0/2}$ obtaining a set $E_\varepsilon$ such that $F:=E_\varepsilon\cap B_{r_0}$ satisfies
\[
\abs{F}_f=\abs{E}_f,\quad \abs{\E(E_\varepsilon)-\E(E)}<C\varepsilon,
\]
for some constant $C=C(n,E,r_0,f,g)>0$. From
\[
\E(F)=\G(E_\varepsilon)-G(r)+\Pe(E_\varepsilon)-P(r)-V'(r)\leq \E(E)+C\varepsilon-(P(r)+G(r))-V'(r),
\]
combined with the optimality of $E$ and Equation \eqref{eq:estimate_growth_GP2}, we get that for all $r>r_0$
\begin{equation}
    -V'(r)\geq -C\varepsilon+G(r)+P(r)\geq C'\varepsilon^{\frac{n-1}{n}},
\end{equation}
provided $\varepsilon>0$ is small enough. Since this lower bound is uniform in $r>r_0$, it implies in particular that there exists a finite radius $r'>0$ such that $V(r')=0$, implying that $E$ is bounded up to a negligible set, as wished.
\end{proof}
We are now ready to discuss the regularity of optimal sets.
\begin{theorem}\label{thm:regularity}
    Let $\psi,g$ be admissible weights of class $C^{k, \alpha}$ and $C^{k+1,\alpha}$ respectively for some $k \geq 1$ and $\alpha \in(0,1]$. Then the boundary of any set of finite perimeter minimizing \eqref{eq:optimization_problem} is of class $C^{k+1, \alpha}$ away from for a singular set of Hausdorff dimension at most $n-8$.  
\end{theorem}
\begin{proof}
It suffices to argue when $k=1.$ Higher order regularity follows from Schauder estimates (cf. Prop 3.3 in \cite{morgan-regularity}). By the boundedness of the minimizer $E \subset B_{R_0}$ and $g\in L^{\infty}_{\operatorname{loc}}$ we immediately deduce that
\begin{align}
\min _{z \in \overline{B_\rho}(x)} e^{\psi(|z|)} \mathcal{P}\left(E ; B_\rho(x)\right) & \leq \mathcal{P}_f\left(E ; B_\rho(x)\right) \leq \mathcal{P}_f\left(F, B_\rho(x)\right)+C_0 \rho^n \\
& \leq \max _{z \in \overline{B_\rho}(x)} e^{\psi(|z|)} \mathcal{P}\left(F, B_\rho(x)\right)+C_0 \rho^n    
\end{align}
where $F$ is a set of finite perimeter, $E\Delta F \subset B_\rho(x)\subset B_{2R_0}$ and the constant $C_0$ depends only on $R_0$, $g$ and $n.$ Now arguing as in the proof of Lemma 5.4 in \cite{fusco-manna} we can conclude that $E$ is an $\omega$-minimizer of the perimeter in $B_{2R_0}$ since
\begin{align}
\mathcal{P}\left(E, B_\rho(x)\right) \leq C \rho^{n-1} \text{ and }\operatorname{osc}_{z \in B_\rho(x)} e^{\psi(|z|)} \leq C \rho.
\end{align}
for some constants depending on $R_0.$ Thus from Remark 21.9 in \cite{maggi-book} we deduce that $\partial^* E$ is $C^{1,\gamma}$ for $\gamma \in (0,1)$ and is $\mathcal{H}^{n-1}$ equivalent to $\partial E$. Note that this is stronger than being $(\Lambda,r_0)$-minimizer as defined in Chap. 21 in \cite{Maggi}. Then arguing as in the proof of Theorem 27.5 in \cite{Maggi} we see that $g\in C^{2,\alpha}$ and $\psi \in C^{1,\alpha}$ implies that $\partial^* E$ is $C^{2,\beta}$ surface for $\beta \in (0,1).$ Finally, the Hausdorff dimension bound on the singular set follows from Theorem 28.1 in \cite{Maggi}.
\end{proof}
After establishing the existence, regularity, and boundedness of solutions, we are ready to discuss the geometric properties of stable and stationary surfaces. To do so, we start by introducing the right notion of mean curvature associated with the first variation of $\E$.
\begin{definition}\label{def:weighted_mean_curvature}
    Let $E\subset \R^n$ be a domain whose boundary has unaveraged mean curvature $H$ and normal $\nu$. For functions $\psi,g:\R^n\to\R$ (not necessarily radial) we call
    \begin{equation}\label{eq:weighted_mean_curvature}
    \Ha:=H+\nabla\psi\cdot\nu+g,
    \end{equation}
    the \emph{weighted mean curvature} of $E$.
\end{definition}
The properties arising from the first variations of the energy with respect to volume-preserving deformations are summarized in the following Lemma.
\begin{lemma}\label{lem:first_variations}
Let $X$ be a vector field defined in a neighborhood of $E\subset \R^n$, and let $\nu$ be the unit normal of $\partial E$, $u=X\cdot \nu$, and $(\phi_t)_{t\in(-\varepsilon,\varepsilon)}$ be the flow of $X$. Then
\begin{equation}\label{eq:first_volume}
\abs{\phi_t(E)}_f=\abs{E}_f+t\int_{\partial E}uf\,d\Haus^{n-1}+o(t),
\end{equation}
and
\begin{equation}\label{eq:first_evergy}
\E(\phi_t(E))=\E(E)+t\int_{\partial E}(H+\nabla \psi\cdot \nu+g) uf\,d\Haus^{n-1}+o(t).
\end{equation}
In particular, if $E$ a critical point of \eqref{eq:optimization_problem}, then there exists a Lagrange multiplier $c\in\R$ such that on $\partial E$
\[
\Ha=H+\nabla \psi\cdot\nu+g=c,
\]
that is, $E$ has constant weighted mean curvature.
\end{lemma}
\begin{proof}
See Chapter 3 in ~\cite{bayle2003proprietes}, Chapter 17.3 in ~\cite{Maggi}, and Section 3 in ~\cite{rosales2008isoperimetric} where the first variation of the perimeter functional is treated. Our result follows by adding the elementary first variation of the weighted volume.
\end{proof}
The second-order variations of the energy take the following form.
\begin{lemma}\label{lem:second_variation}Under the same assumptions of Lemma, \ref{lem:first_variations}, suppose additionally that $E$ has constant weighted mean curvature. Then
\begin{equation}\label{eq:second_variation}
\begin{split}
Q(u,u)&:=\frac{d^2}{dt^2}\Big\vert_{t=0}\Bigl(\E(\phi_t(E))-\Ha\abs{\phi_t(E)}_f\Bigr)\\
&=-\int_{\partial E}\Bigl(\Delta u+\nabla\psi\cdot\nabla u+(\abs{\sigma}^2-\nabla^2\psi(\nu,\nu)-\nabla g\cdot\nu)u\Bigr)uf\,d\Haus^{n-1},
\end{split}
\end{equation}
where $\sigma$ is the second fundamental form of $\partial E$, and the operators applied to $u$ are to be considered with respect to the induced metric on $\partial E$. Furthermore, $E$ is stable (that is minimizes $\E$ up to the second order under volume constraint) if and only if $Q\geq 0$ for every variation infinitesimally preserving the volume, that is 
\[
\abs{\phi_t(E)}_f=\abs{E}_f+o(t),
\]
for all $t\in(-\varepsilon,\varepsilon)$.
\end{lemma}
\begin{proof}
As is the proof of Lemma \ref{lem:first_variations}, we refer to \cite{barbosa2012stability}, Section 3.4.6 \cite{bayle2003proprietes}, and Proposition 2.10 in \cite{bayle2003proprietes}.
\end{proof}
\begin{remark}
    Equations \eqref{eq:first_volume}, \eqref{eq:first_evergy}, and \eqref{eq:second_variation} hold when the ambient space is a general weighted manifold $(M,f=e^\psi)$. The correct way to express the second variation is $Q(u,u)=-\inn{Lu,u}_{L^2(\partial E)}$, with $L$ the Jacobian operator 
    \[
    L:=\Delta_f+\abs{\sigma}^2+\Ric_f(\nu,\nu)-\nabla g\cdot\nu,
    \]
    where $\sigma$ the second fundamental form of $\partial E$, $\Ric_f:=\Ric-\nabla^2\psi$ the Bakri-\'Emery Ricci tensor, and $\Delta_f:=f^{-1}\diva(f\nabla)$ the drifted Laplacian associated to $M$. We refer the interested reader to \cite{bayle2003proprietes,morgan2005manifolds,rosales2021stable,castro2014free}.
\end{remark}
\begin{proposition}
    Suppose that $\psi,g$ are radially symmetric with respect to the same point $y\in\R^n$. Then, any sphere centered in $y$ and of radius $R>0$ has constant weighted mean curvature. Moreover, it is stable in the sense of Lemma \ref{lem:second_variation} if and only if
    \begin{equation}\label{eq:stability_of_centered_spheres}
    \psi''(R)+g'(R)\geq 0,
    \end{equation}
    holds true.
\end{proposition}
\begin{proof}We follow the argument of Lemma 3.2 in~\cite{rosales2008isoperimetric}, including here the adapted proof for the sake of completeness. Without loss of generality set $y=0$. Let $\partial B_R$ be a sphere of radius $R>0$ centered at the origin. Then, by Definition \ref{def:weighted_mean_curvature} we have that its weighted mean curvature is equal to
    \[
    \Ha=\frac{n-1}{R}+\psi'(R)+g(R),
    \]
    which is constant. By Lemma \ref{lem:second_variation} we can compute for any volume preserving variation $u$ 
\begin{equation}\label{eq:variation_on_sphere}
    Q(u,u)=-f(R)\int_{\partial B_R}(\Delta u+\abs{\sigma}^2u)u\,d\Haus^{n-1}+f(R)\Bigl(\psi''(R)+g'(R)\Bigr)\int_{\partial B_R}u^2\,d\Haus^{n-1}.
\end{equation}
Now, notice that in this particular case any perturbation $u$ preserves the weighted volume $\abs{B_R}_f$ if and only if it preserves the unweighted volume $\abs{B_R}$ in the sense of Equation \eqref{eq:first_volume}, since
\[
\int_{\partial B_R}u f\,d\Haus^{n-1}=f(R)\int_{\partial B_R}u\,d\Haus^{n-1}.
\]
Hence, if $\psi''(R)+g'(R)\geq 0$, then $Q$ is bounded from below by the unweighted second variation of $B_R$, which we know to be positive. Conversely, for any $\xi\in S^{n-1}$ the variation $w=\xi\cdot \nu$ represents a translation in the $\xi$ direction, and
\[
Q(w,w)=f(R)(\psi''(R)+g'(R))\int_{\partial B_R}w^2\,d\Haus^{n-1},
\]
proving that if $\partial B_R$ is stable, then $\psi''(R)+g'(R)\geq 0$.
\end{proof}
\subsection{In one dimension}\label{sec:one-d}In this section we prove Theorem \ref{thm:one_dim}, completely addressing Problem \eqref{eq:optimization_problem} in $\R$.

Let $\psi,g$ be admissible weights in Definition \ref{def:admissible_weights}. It will be convenient in the sequel to introduce the function
\begin{equation}\label{eq:kappa}
\kappa(x):=\psi''(x)+g'(x),
\end{equation}
which by assumption defines a positive continuous function for every $x\geq 0$. In the one-dimensional case, an interval $(a,b)$ has weighted volume and energy
\[
\abs{(a,b)}_f=\int_a^b f\,dx,\quad \E((a,b))=f(a)+f(b)+\int_a^b g(\abs{x})f(\abs{x})\,dx.
\]
Integrating Equation \eqref{eq:kappa} we get
\[
g(x)=g(0)+\int_0^x\kappa(t)\,dt-\psi'(x).
\]
Set $K(x):=\int_0^x\kappa(t)\,dt$, and notice that by assumption $K$ is monotone increasing in $[0,+\infty)$. Similarly, introduce the maps
\[
F(x):=\int_0^xf\,dt,\quad H(v):=F^{-1}(v).
\]
Notice that the map $H:[0,F(+\infty))\to[0,+\infty)$ is well defined since $F$ is itself monotone increasing. Taking advantage of the explicit form of $g$ in terms of $K$ and $\psi'$, we can explicitly compute the energy for a general interval.
\begin{lemma}
    Let $(a,b)\subset(0,+\infty)$ be an interval of volume $\abs{(a,b)}_f=v$. Then,
    \[
    \E((a,b))=2f(a)+\int_{F(a)}^{F(a)+v}K(H(w))\,dw+g(0)v.
    \]
    Let $(-a,b)\subset$ be an interval containing the origin, and set $v^+:=\abs{(0,a)}_f$ and $v^-:=\abs{(0,b)}_f$. Then,
    \[
    \E((-a,b))=2f(0)+\int_{0}^{v^+}K(H(w))\,dw+\int_{0}^{v^-}K(H(w))\,dw+g(0)(v^++v^-).
    \]
\end{lemma}
\begin{proof}
    This is a direct computation: in the first situation
    \begin{align*}
        \E((a,b))&=f(a)+f(b)+\int_a^bg(\abs{x})f\,dx\\
        &=f(a)+f(b)+\int_a^b(g(0)+K(x)-\psi')e^{\psi}\,dx\\
        &=f(a)+f(b)+g(0)v+\int_a^bKf\,dx-(f(b)-f(a))\\
        &=2f(a)+g(0)v+\int_a^b Kf\,dx.
    \end{align*}
    Now, since $\abs{(a,b)}_f=v$ we get immediately that $F(b)-F(a)=v$, so that
    \begin{align*}
    \int_a^bKf\,dx&=\int_{H(F(a))}^{H(F(a)+v)}Kf\,dx=\int_{F(a)}^{F(a)+v}K(H(w))f(H(w))H'(w)\,dw\\
    &=\int_{F(a)}^{F(a)+v}K(H(w))F'(H(w))H'(w)\,dw\\
    &=\int_{F(a)}^{F(a)+v}K(H(w))\,dw.
    \end{align*}
    as wished. To compute the energy of the interval containing the origin $(-a,b)$ we simply notice that
    \begin{align*}
    \E((-a,b))&=\E((0,a))+\E((0,b))-2f(0)\\
    &=2f(0)+\int_{0}^{v^+}K(H(w))\,dw+\int_{0}^{v^-}K(H(w))\,dw+g(0)(v^++v^-),
    \end{align*}
    taking advantage of the computation we performed before.
\end{proof}
We start by proving one implication in Theorem \ref{thm:one_dim}.
\begin{proposition}\label{prop:oneD}
    Let $n=1$ and $\psi,g$ be admissible weights so that $\min_{x\geq 0}\psi(x)=\psi(0)$. Then, centered symmetric intervals minimize \eqref{eq:optimization_problem}. If moreover $\psi,g$ are \emph{strictly} admissible, then centered intervals \emph{uniquely} minimize \eqref{eq:optimization_problem}.
\end{proposition}
\begin{proof}
    Let $\psi,g$ be admissible weights. Let $\Omega$ be a generic finite union of intervals. Then, $\Omega$ can be expressed in the form
    \[
    \Omega=[-b_{-M},-a_{-M}]\cup\dots\cup[-b_{-1},-a_{-1}]\cup [a_1,b_1]\cup\dots\cup[a_N,b_N],
    \]
    where $a_j,b_j\geq 0$, all the intervals are supposed to be mutually disjoint except the central one $[-b_{-1},-a_{-1}]\cup [a_1,b_1]$ that form an interval containing the origin in case $-a_{-1}=a_1=0$. Let $v=\abs{\Omega}_f$ and $v_j$ the respective volumes of the intervals constituting $\Omega$, so that $v_{-M}+\dots+v_{N}=v$. Set $v^-:=\abs{\Omega\cap (-\infty,0)}_f$ and $v^+:=\abs{\Omega\cap (0,+\infty)}_f$. We claim that the interval
    \[
    J:=[-H(v^-),H(v^+)],
    \]
    satisfies $\E(J)\leq\E(\Omega)$. The fact that $\abs{J}_f=\abs{\Omega}_f$ is clear by definition of $H$, $v^-$, and $v^+$. Now,
    \begin{align*}
        \E(\Omega)=\sum_{j=-M,j\neq 0}^N\Bigl(2f(a_j)+\int_{F(a_j)}^{F(a_j)+v_j}K(H(w))\,dw\Bigr)-2f(0)\delta_{\{a_{-1}=a_1=0\}}+g(0)v,
    \end{align*}
    and
    \[
    \E(J)=2f(0)+\int_0^{v^-}K(H(w))\,dw+\int_0^{v^+}K(H(w))\,dw+g(0)v.
    \]
    Since $K\circ H$ is monotone increasing, it follows that (see \cite[Claim 2]{indrei2025one})
    \begin{equation}\label{eq:monotone_KoH}
    \sum_{j=-M,j\neq 0}^N\int_{F(a_j)}^{F(a_j)+v_j}K(H(w))\,dw\geq\int_0^{v^-}K(H(w))\,dw+\int_0^{v^+}K(H(w))\,dw.
    \end{equation}
    By assumption $f(a_j)\geq f(0)$ for all $a_j$, showing that $\E(J)\leq\E(\Omega)$, as claimed. We prove that among all intervals of the same volume $v$ containing the origin, the symmetric one is a minimizer of the energy. For that, define for every $\lambda\in[0,1]$ the interval
    \[
    J_\lambda:=[H(\lambda v),H((1-\lambda)v)].
    \]
    The fact that $\E(J_{1/2})\leq\E(J_\lambda)$ for all $\lambda\neq 1/2$, follows form observing that the map $\Psi:\lambda\mapsto \E(J_\lambda)$ is symmetric $\Psi(\lambda)=\Psi(1-\lambda)$ and convex. Symmetry is immediate noticing that
    \[
    \Psi(\lambda)=2f(0)+\int_0^{v}K(H(\lambda w))\lambda+K(H((1-\lambda) w))(1-\lambda)\,dw+g(0)v.
    \]
    Now, setting $K\circ H=h$ for simplicity of exposition, we can compute
    \begin{align*}
        \frac{d\Psi}{d\lambda}&=\frac{d}{d\lambda}\int_0^v h(\lambda w)\lambda+h((1-\lambda) w)(1-\lambda)\,dw\\
        &=\int_0^vh'(\lambda w)\lambda w+h(\lambda w)-h'((1-\lambda)w)(1-\lambda)w-h((1-\lambda)w)\,dw\\
        &=\int_0^v\frac{d}{dw}(h(\lambda w)w)-\frac{d}{dw}(h((1-\lambda)w)w)\,dw\\
        &=h(\lambda v)v-h((1-\lambda)v)v,
    \end{align*}
    and so
    \begin{equation}\label{eq:double_derivative_psi}
    \frac{d^2\Psi}{d\lambda^2}=h'(\lambda v)v^2+h'((1-\lambda)v)v^2\geq 0,
    \end{equation}
    proving that
    \[
    \E(\Omega)\geq\E(J)\geq\E(J_{1/2}),
    \]
    as wished. If moreover $\psi,g$ are \emph{strictly} admissible weights, then $h=K\circ H$ is \emph{strictly} monotone increasing. As a consequence, one gets strict inequality in \eqref{eq:monotone_KoH} unless $\Omega$ is an interval containing the origin, and similarly from \eqref{eq:double_derivative_psi} we deduce $\E(J)>\E(J_{1/2})$ unless $\Omega=J_{1/2}$, showing uniqueness of minimizers in this case.
 \end{proof}
The next proposition will imply together with Proposition \ref{prop:oneD} the missing implication of Theorem \ref{thm:one_dim}. The failure of the optimality of centered intervals relies on the fact that in the small volume regime, optimizers tend to concentrate around global minima of the weight $e^\psi$, since for small volumes the weighted perimeter $\Pe$ is the dominating term in the energy $\E$.
\begin{proposition}\label{prop:thmonedir}
    Let $n=1$ and $\psi,g$ be admissible weights. Suppose there exists $x_0>0$ such that $\psi(x_0)<\psi(0)$. Then, there exists $v_0>0$ so that the interval of volume $v\in(0,v_0)$ with left endpoint in 
    $x_0$ has strictly less energy than the centered symmetric interval of the same volume.
\end{proposition}
\begin{proof}
    Fix any volume $v\in(0,v_0)$, with $v_0$ yet to define, and let $J_0=[-H(v/2),H(v/2)]$ and $J_1=[x_0,H(F(x_0)+v)]$ be intervals, so that $\abs{J_0}_f=\abs{J_1}_f=v$.
    Then
    \begin{align*}
        \E(J_1)-\E(J_0)=2\Bigl(f(x_0)-f(0)\Bigr)+\Bigl(\int_{F(x_0)}^{F(x_0)+v}K(H(w))\,dw-2\int_{0}^{v/2}K(H(w))\,dw\Bigr).
    \end{align*}
    By continuity, there exists $v_0$ small enough so that
    \[
    \Bigl(\int_{F(x_0)}^{F(x_0)+v}K(H(w))\,dw-2\int_{0}^{v/2}K(H(w))\,dw\Bigr)<-(f(x_0)-f(0)),
    \]
    showing
    \[
    \E(J_1)-\E(J_0)<f(x_0)-f(0)<0,
    \]
    as wished.
\end{proof}
\begin{remark}\label{rmk:unique_min}
This is true also in arbitrary dimensions: in $\R^2$ one can check that a ball centered in $(r_0,0)$ with volume $v\ll 1$ has energy $\E=2\sqrt{\pi f(r_0)v}+O(v)$.
\end{remark}

\begin{proof}[Proof of Theorem \ref{thm:one_dim}]
    Combine Proposition \ref{prop:oneD} with Propositon \ref{prop:thmonedir}. 
\end{proof}

\subsection{Large volume regime}\label{sec:large-volume}For now on, we will consider $n\geq 2$.

We prove that centered balls uniquely minimize \eqref{eq:optimization_problem} if $v>0$ is big enough and $\psi$, $g$ are $\kappa$-uniformly admissible in the sense of Definition \ref{def:admissible_weights}. The proof, which follows the lines of \cite{kolesnikov2011isoperimetric}, is based on the following lemma, and a suitable calibration argument.
\begin{lemma}\label{lem:levelset}
    Let $\mu$ a Borel measure on $\R^n$ and $h:\R^n\to[0,+\infty)$ a measurable function so that setting $C_t=\{x:h(x)\leq t\}$, the map $t\mapsto\mu(C_t)$ is continuous and strictly increasing from 0 to $\mu(\R^n)$. Then, for every Borel set $A\subset\R^n$ one has that
    \[
    \int_A h \,d\mu\geq\int_{C_t}h\,d\mu,
    \]
    where $t>0$ is such that $\mu(A)=\mu(C_t)$.
\end{lemma}
\begin{proof}
    We refer to Lemma 6.4 in~\cite{kolesnikov2011isoperimetric}. The proof is elementary:
    \begin{align*}
    \int_A h\,d\mu&=\int_{A\cap C_t} h\,d\mu+\int_{A\setminus C_t} h\,d\mu\geq \int_{A\cap C_t} h\,d\mu+t\mu(A\setminus C_t)=\int_{A\cap C_t} h\,d\mu+t\mu(C_t\setminus A)\\
    &\geq \int_{C_t}h\,d\mu,
    \end{align*}
    where the identity $\mu(A\setminus C_t)=\mu(C_t\setminus A)$ follows form the assumption $\mu(A)=\mu(C_t).$
\end{proof}
\begin{proof}[Proof of Theorem \ref{thm:large_vol}]Let $X=\ell(r)\frac xr$ be a radial vector field, where $\ell:[0,+\infty)\to[0,1]$ is some $C^2$ increasing function yet to define. Then, for any set $F$  of finite perimeter we can estimate
    \begin{align*}
        \E(F)&=\int_{\partial^* F}f\,d\Haus^{n-1}+\int_F fg \,dx\geq \int_{\partial^*F}fX\cdot\nu\,d\Haus^{n-1}+\int_F \ell fg\,dx\\
        &=\int_F\diva(fX)+\ell fg\,dx=\int_F f\Bigl(\ell '+\ell(\psi'+\frac{n-1}{r})\Bigr)+\ell fg\,dx\\
        &=\int_F f\Bigl(\ell '+\ell\Bigl(\psi'+g+\frac{n-1}{r}\Bigr)\Bigr)\,dx=\int_F hf\,dx,
    \end{align*}
    where $h:=\ell '+\ell(\psi'+g+\frac{n-1}{r})$. If we can find peculiar radius $r^*>0$ and function $\ell$ such that $\ell(r)=1$ for every $r\geq r^*$ and $h\geq 0$, $h'\geq 0$, we can conclude thanks to Lemma \ref{lem:levelset} that centered geodesic balls of radius $r\geq r^*$ are optimal, simply by noticing that the above inequality is, in fact, an equality in this particular case. Suppose first $\kappa=1$. To find a suitable $\ell$, notice that
    \[
    h=\ell '+\ell\Bigl(\psi'+g+\frac{n-1}{r}\Bigr)\geq\ell'+\ell r+\frac{n-1}{r}\geq 0
    \]
    and
    \[
    h'=\ell''+\Bigl(\ell\frac{n-1}{r}\Bigr)'+\ell'(\psi'+g)+\ell(\psi''+g')\geq\ell''+\Bigl(\ell\frac{n-1}{r}\Bigr)'+\ell'r+\ell.
    \]
    One can check that the function
    \[
    \ell(r)=\begin{cases}\frac3{2\sqrt{n+2}}\Bigl(r-\frac{r^3}{3(n+2)}\Bigr),&0<r\leq \sqrt{n+2},\\1,& r>\sqrt{n+2}.\end{cases}
    \]
   is suitable (see Corollary 6.8 in~\cite{kolesnikov2011isoperimetric}). The case $\psi''+g'\geq\kappa>0$ for $\kappa>0$ arbitrary can be reduced to the case $\kappa=1$ via the dilation $y=\lambda x$, with $\lambda=\sqrt{\kappa}$. In fact, setting $s=\abs{y}=\lambda r$, $\tilde F=\lambda F$, $\tilde f(s)=f(s/\lambda)$, and $\tilde g(s)=g(s/\lambda)/\lambda$, one has that $\lambda^{n-1}\E(F)=\tilde \E(\tilde F)$ and $\lambda^n\abs{F}_f=\abs{\tilde F}_{\tilde f}$, where $\tilde \E$ and $\abs{\cdot}_{\tilde f}$ are the functionals associated to $\tilde f$ and $\tilde g$, and writing $\tilde f=e^{\tilde \psi}$ one gets
    \[
    \tilde\psi''+\tilde g'=\lambda^{-2}\bigl(\psi''+g'\bigr)(s/\lambda)\geq 1.
    \]
    Hence, in the general case, one can choose $r^*=\sqrt{\frac{n+2}{\kappa}}$.
\end{proof}

\subsection{Counter-examples in higher dimensions}\label{subsec:counterexamples}
In this section we prove that in dimension higher or equal to two, Theorem \ref{thm:one_dim} fails to be true, meaning that, under the admissible conditions of Definition \ref{def:admissible_weights}, it is not sufficient for $\psi$ to have a global minimum at the origin to ensure minimality of centered balls for all volumes. The difficulty in constructing counterexamples relies on the fact that for small volumes minimizers of \eqref{eq:optimization_problem} will, in fact, concentrate around the origin (see Remark \ref{rmk:unique_min}) making the analysis difficult, and for large volumes, we know by Theorem \ref{thm:large_vol} proved in the previous section that spheres are in fact optimal. Hence, we are bound to search for counterexamples in the intermediate volume regime, in which the two terms composing the energy $\E=\Pe+\G$ have the same importance, and hence they have to be treated carefully in order to take advantage of their interactions.

Our goal is to construct $\kappa$-uniformly admissible weights so that $\psi$ has a unique minimum in zero but for certain volumes centered spheres are not optimal. More precisely, we split the proof of Proposition \ref{prop:counter_examples} into two propositions: in Proposition \ref{prop:g_monotone} we impose $g'\geq 0$, and in Proposition \ref{prop:psi_monotone} we impose $\psi'\geq 0$.

\begin{proposition}\label{prop:g_monotone}
    In $\R^n$, $n\geq 2$, there exist $\kappa$-uniformly admissible weights $\psi$, $g$ so that $g'\geq 0$, and $\psi$ has a unique minimum at the origin, for which centered balls do not always minimize \eqref{eq:optimization_problem}.
\end{proposition}
\begin{proof}
    We first construct $\psi$ and $g$ as admissible weights in the sense of Definition \ref{def:admissible_weights}, that is $\psi''+g'\geq 0$. Fix $M>0$ and fix some volume $v>0$ and lengths $0<L<L'<L''$ so that $v<\min\{\omega_n e^ML^n/2,\omega_ne^{M/2}(L'-L'')^n\}$. Let $\varepsilon>0$ be small enough in terms of $v$, in this situation it will be sufficient that
    \begin{equation}\label{eq:small_epsilon}
    \varepsilon^{n-1}\leq \frac14e^{-M}\omega_n^{\frac{1-n}n}v^{\frac{n-1}{n}},
    \end{equation}
    and suppose that $h=L'-L$ is big enough in terms of $v$ and $\varepsilon$, say 
    \begin{equation}\label{eq:big_h}
    h\geq\frac{v}{\omega_n}e^{-M}\varepsilon^{1-n}.
    \end{equation}
    \begin{figure}[htbp]
\centering
\includegraphics[scale=1]{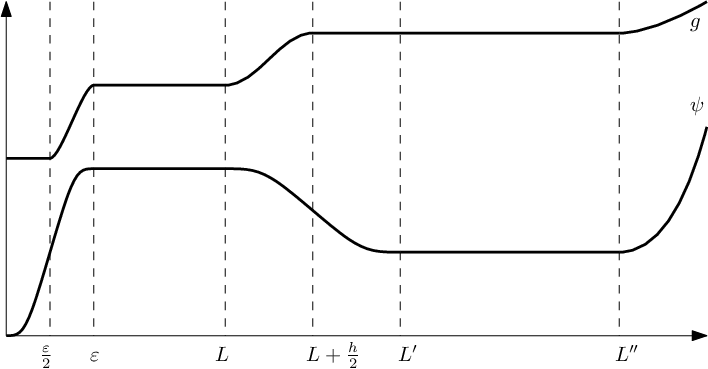}
\caption{First construction of $\psi$ and $g$.}
\label{fig:psi}
\end{figure}

    We construct $\psi$ as in Figure \ref{fig:psi}, that is
    \begin{itemize}
        \item[--] $ \psi(0)=\psi'(0)=0$,
        \item[--] $\psi=M$ on $[\varepsilon,L]$, and $\psi={M/2}$ on $[L',L'']$,
        \item[--] $0\leq\psi'(r)\leq 2M/\varepsilon$ on $[0,\varepsilon]$,
        \item[--] $0\leq-\psi'(r)\leq M/h$ in $[L,L']$,
        \item[--] $\psi$ convex in $[0,\varepsilon/2]$, $[L+h/2,L']$, and $[L'',+\infty)$,
        \item[--]There exists $c>0$ such that $\psi''>c>0$ in $[0,\varepsilon/4]$
        \item[--] $\psi$ concave in $[\varepsilon/2,\varepsilon]$ and $[L,L+h/2]$.
    \end{itemize}
    We then set for some constant $C>0$
    \[
    g(r):=C-\int_0^r\min\{\psi''(t),0\}\,dt,
    \]
    for $r\in[0,L'']$ and $g'(r)=1$. Clearly, for $r\leq L''$ one has that
    \[
    g'(r)=-\min\{\psi''(r),0\}+\min\{\psi''(0),0\}\geq -\psi''(r),
    \]
    so that $\psi''+g'\geq 0$ in $[0,L'']$, proving that $\psi,g$ are admissible weights. Moreover, taking $C>0$ big enough, we can enforce $g>0$ everywhere. Call $B_\varepsilon$ the ball centered at the origin with radius $\varepsilon>0$. Notice that with the choice of the constants we did at the beginning, the ball $B$ centered at the origin with weighted volume $v$ has radius $\tau>0$ strictly between $\varepsilon$ and $L$, since
    \[
    \abs{B_\varepsilon}_f<\omega_ne^M\varepsilon^n\leq 4^{-\frac n{n-1}}e^{-\frac M{n-1}}v<v,
    \]
    and
    \[
    \abs{B_L}_f\geq \omega_n e^M(L^n-\varepsilon^n)>2v-\omega_n e^M\varepsilon^n\geq \Bigl(2-e^{-\frac{M}{n-1}}4^{-\frac{n}{n-1}}\Bigr)v>v.
    \]
    Moreover, the ball $B'$ distant $L'$ from the origin and with radius $\tau'=(e^{-M/2}v/\omega_n)^{1/n}$ has the same weighted volume as $B$ and is completely contained in the annulus $\{\abs{x}\in[L',L'']\}$. We show that
    \[
    \E(B')<\E(B).
    \]
    By definition
    \begin{align*}
        \E(B)-\E(B')&=\Pe(B)-\Pe(B')+\G(B)-\G(B')\\
        &=n\omega_n(e^M\tau^{n-1}-e^{M/2}(\tau')^{n-1})+\G(B)-\G(B').
    \end{align*}
    Since $v=\abs{B}_f<\omega_ne^M\tau^n$, we obtain that
    \[
        \Pe(B)-\Pe(B')>n\omega_n^{\frac 1n}v^{\frac {n-1}{n}}(e^{\frac{M}{n}}-e^{\frac{M}{2n}})\geq \omega_n^{\frac1n}v^{\frac{n-1}{n}}M,
    \]
  since $e^x-e^{\frac x2}\geq\frac x2$ for all $x\in\R$. Recalling that in $[0,\varepsilon]$ we imposed $\psi'\leq 2M/\varepsilon$, we have that for $r\in[0,\varepsilon]$ one has that
  \begin{align*}
  g(r)-g(\varepsilon)&=C-\int_0^r\min\{\psi'',0\}\,dt-C+\int_0^\varepsilon\min\{\psi'',0\}\,dt\\
  &=\int_{\max\{r,\varepsilon/2\}}^\varepsilon\min\{\psi'',0\}\,dt\\
  &=\psi'(\varepsilon)-\psi'(\max\{r,\varepsilon/2\})\geq-\frac{2M}{\varepsilon}.
  \end{align*}
  Similarly, in $[L',L'']$ we can estimate
  \begin{align*}
      g(r)&=g(L')=g(\varepsilon)-\int_L^{L+h/2}\psi''\,dt=g(\varepsilon)-\psi'(L+h/2)+\psi'(L)\\
      &\leq g(\varepsilon)+\frac{M}{h},
  \end{align*}
  since we assumed $0\leq-\psi'(r)\leq M/h$ in $[L,L']$. We are ready to compute the gap between the potential energies of $B$ and $B'$:
\begin{align*}
\G(B)-\G(B')&\geq \int_B gf\,dx-\int_{B'} gf\,dx\\
&\geq \int_B g(\varepsilon)f\,dx+\int_{B_\varepsilon}(g-g(\varepsilon))f\,dx-\Bigl(g(\varepsilon)+\frac Mh\Bigr)v\\
&=\int_{B_\varepsilon}(g-g(\varepsilon))f\,dx-\frac{M}{h}v\\
&\geq -\frac{2M}{\varepsilon}\abs{B_\varepsilon}_f-\frac{M}{h}v\\
&\geq -2Me^M\varepsilon^{n-1}\omega_n-\frac{M}{h}v\\
&\geq -3Me^M\varepsilon^{n-1}\omega_n,
\end{align*}
where in the last line we took advantage of the bound on $h$ in Equation \eqref{eq:big_h}.
Hence, by combining the estimates, we finally get that
    \begin{equation}\label{eq:gap_counter1}
        \E(B)-\E(B')>\omega_n^{\frac1n}v^{\frac{n-1}{n}}M-3\omega_nMe^M\varepsilon^{n-1},
    \end{equation}
    which is strictly greater than zero provided
    \[
    \varepsilon^{n-1}<\frac14e^{-M}\omega_n^{\frac{1-n}{n}}v^{\frac{n-1}{n}},
    \]
    as we imposed in Equation \eqref{eq:small_epsilon}, proving that $B'$ has strictly less energy than the centered ball $B$, as wished. 
    
    We are left to argue that one can in fact construct $\psi,g$ not only admissible weights but $\kappa$-uniformly admissible, meaning that there exists $\kappa>0$ such that $\psi''+g'\geq \kappa$ everywhere, as in Definition \ref{def:admissible_weights}. This is achievable via adding a small enough linear term to $g$ (the same can be achieved by adding a small enough quadratic term to $\psi$, but it will change also the volumes of $B$ and $B'$, making the analysis more involved). Fix $\delta>0$, and let $\alpha(r)$ be any monotone, smooth function such that $\alpha'(0)=0$, $\alpha(r)=\delta r$ for $r\geq \varepsilon/4$. Then, letting $\tilde g:=g+\alpha$, we have clearly that $\tilde g$ is monotone and
    \[
    \psi''+\tilde g\geq\kappa>0
    \]
    with $\kappa:=\min\{\min\{\psi''(r):r\in[0,\varepsilon/4]\},\delta\}\geq\min\{c,\delta\}$. Hence, $\psi$, $\tilde g$ are $\kappa$-uniformly admissible weights. Now, we estimate the error we make in computing the potential term as
    \begin{align*}
    \Bigl(\int_B\tilde gf \,dx-\int_{B'}\tilde gf\,dx\Bigr)&-(\G(B)-\G(B'))=\int_B\alpha f \,dx-\int_{B'}\alpha f\,dx\geq -\delta L''v.
    \end{align*}
    Hence
    \begin{align*}
        \Bigl(\int_B\tilde gf \,dx-\int_{B'}\tilde gf\,dx\Bigr)&+(\Pe(B)-\Pe(B'))\geq \E(B)-\E(B')-\delta L'' v\\
        &\geq \frac{1}{4}\omega^{\frac1n}v^{\frac{n-1}{n}}M-\delta L'' v,
    \end{align*}
    which is positive provided $\delta>0$ is small enough, completing the proof.
\end{proof}
We construct now a second counter-example which prescribes $\psi$ monotone increasing.
\begin{proposition}\label{prop:psi_monotone}
    In $\R^n$, $n\geq 2$, there exist $\kappa$-uniformly admissible weights $\psi$, $g$ so that $\psi'\geq 0$, and $\psi$ has a unique minimum at the origin, for which centered balls do not always minimize \eqref{eq:optimization_problem}.
\end{proposition}
\begin{proof}
\begin{figure}[htbp]
\centering
\includegraphics[scale=1]{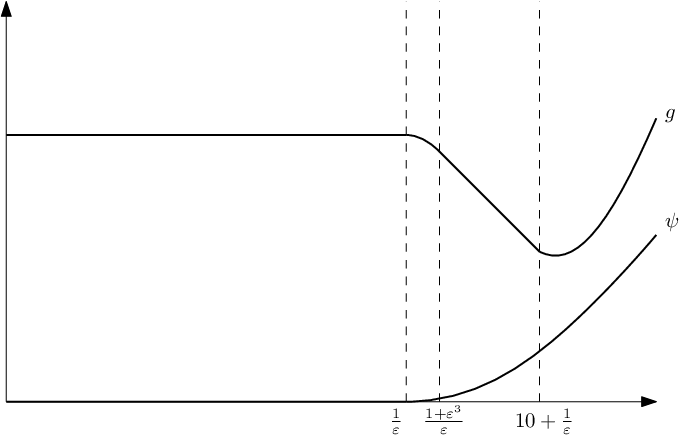}
\caption{Construction of $\psi$.}
\label{fig:psi2}
\end{figure}
 It is enough to construct $\psi$ and $g$ to be admissible weights, since arguing exactly as the end of the proof of Proposition \ref{prop:g_monotone}, it is sufficient to add a small enough linear term to $g$ in order ensure $\kappa$-uniformly admissibility.
 
 We construct $\psi$ as in the Figure  \ref{fig:psi2}: for $\varepsilon>0$ small enough, we set $\psi(r)=0$ for $r\in[0,\varepsilon^{-1}]$, $\psi(r)=\varepsilon(r-\varepsilon^{-1})^2$ for $r>\varepsilon^{-1}+\varepsilon^2$, and in $r\in[\varepsilon^{-1},\varepsilon^{-1}+\varepsilon^2]$ we set $\psi$ to be any convex function connecting the two prescribed parts in a $C^2$-fashion. For $g(0)$ big  enough, we prescribe $g(r):=g(0)-\psi'(r)$ for $r\in [0,10+\varepsilon^{-1}]$, and $g(r)\to+\infty$  for $r$ outside this interval. Doing so, we ensure $g>0$ and $\psi''+g'\geq 0$. Select an arbitrary basis $(x_1,\dots,x_n)$ of $\R^n$, and let $B'$ a ball of radius $R<1$ whose center lies in the $x_1$-axis, at coordinate $\varepsilon^{-1}+R$. Let $B$ be the centered ball with the same weighted volume: $\abs{B}_f=\abs{B'}_f$, and let $\tau>0$ be its radius. We prove that for every $R<1$ there exists $\varepsilon>0$ such that
\[
\E(B)-\E(B')>0.
\]
First of all, we notice that the Euclidean perimeter $P$ of $B'$ is strictly smaller than $\Pe(B)=P(B)$, since the density is increasing, and hence its Euclidean volume $V$ is decreasing. More precisely, since $e^x\geq 1+x$ for every $x\in\R$, we get that
\[
V(B)=\abs{B}_f=\abs{B'}_f=\int_{B'}e^\psi\,dx\geq V(B')+\int_{B'}\psi\,dx
\]
and therefore
\begin{equation}\label{eq:brutal_bound}
P(B')=n\omega_n^{\frac1n}V(B')^{\frac n{n-1}}\leq n\omega_n^{\frac1n} \Bigl(V(B)-\int_{B'}\psi\,dx\Bigr)^{\frac{n-1}{n}}.
\end{equation}
Letting $\lambda:=\frac{P(B')}{P(B)}\in(0,1)$, we notice that
\begin{align*}
P(B)-P(B')&=(1-\lambda)P(B)=(1-\lambda)P(B)^{\frac{n}{n-1}}P(B)^{-\frac{1}{n-1}}\\
&=\frac{1-\lambda}{1-\lambda^{\frac n{n-1}}}(1-\lambda^{\frac{n}{n-1}})P(B)^{\frac{n}{n-1}}P(B)^{-\frac{1}{n-1}}\\
&=\frac{1-\lambda}{1-\lambda^{\frac n{n-1}}}(P(B)^{\frac{n}{n-1}}-P(B')^{\frac{n}{n-1}})P(B)^{-\frac{1}{n-1}}\\
&>\frac{n-1}{n}P(B)^{-\frac{1}{n-1}}(P(B)^{\frac{n}{n-1}}-P(B')^{\frac{n}{n-1}}),
\end{align*}
which combined with Equation \eqref{eq:brutal_bound} gives us the estimate
\begin{equation}\label{eq:counter_bound_per}
P(B)-P(B')>\frac{n-1}{n}P(B)^{-\frac{1}{n-1}}n^{\frac n{n-1}}\omega_n^{\frac1{n-1}}\int_{B'}\psi\,dx=\frac{n-1}{\tau}\int_{B'}\psi\,dx.
\end{equation}
The error we have committed by replacing $\Pe(B')$ with  $P(B')$ can be estimated as
\begin{align*}
\Pe(B')-P(B')&=\int_{\partial B'}(e^\psi-1)\,d\Haus^{n-1}=\int_{\partial B'}\psi\,d\Haus^{n-1}+\int_{\partial B'}(e^\psi-1-\psi)\,d\Haus^{n-1}\\
&\leq \int_{\partial B'}\psi\,d\Haus^{n-1}+n\omega_nR^{n-1}\sum_{k=2}^{+\infty}\frac{(4R^2\varepsilon)^k}{k!},
\end{align*}
and the potential energy gap is equal to
\[
\G(B)-\G(B')=\int_{B}g(0)\,dx-\int_{B'}(g(0)-\psi')e^\psi\,dx=\int_{B'}\psi'e^\psi\,dx.
\]
Combining the above estimates we get that
\begin{equation}\label{eq:combine_eq_1}
    \E(B)-\E(B')>\frac{n-1}{\tau}\int_{B'}\psi\,dx-\int_{\partial B'}\psi\,d\Haus^{n-1}+\int_{B'}\psi'e^\psi\,dx+o(\varepsilon),
\end{equation}
where here, and in the sequel, we will assume that $o(\varepsilon)$ can be dependent on $R$ (that has been fixed once and for all at the beginning of the argument). Now, we take advantage of the fact that $B'$ was placed far away from the origin so that integrating along $r$ reduces to integrating along $x_1>0$ up to making a reasonably small error. In fact, setting $\hat x_1:=x_1-\varepsilon^{-1}$, we notice that
\[
\hat x_1\leq\Bigl((\hat x_1+\varepsilon^{-1})^2+\sum_{i=2}^n x_i^2\Bigr)^{1/2}-\varepsilon^{-1}=r-\varepsilon^{-1},
\]
and in $B'$
\begin{align*}
r-\varepsilon^{-1}=(\hat x_1+\varepsilon^{-1})\Bigl(1+\frac{\sum_{i=2}^n x_i^2}{(\hat x_1+\varepsilon^{-1})^2}\Bigr)^{1/2}-\varepsilon^{-1}\leq \hat x_1+\frac{1}{2(\hat x_1+\varepsilon^{-1})}\leq \hat x_1+\frac\varepsilon 2,
\end{align*}
where we took advantage of the inequality $\sqrt{1+a^2}\leq 1+a^2/2$, and that $R<1$ by assumption. To bound $\tau$ in terms of $R$, we can estimate from above the volume of $B'$ as
\[
\omega_n\tau^n=\abs{B}_f=\abs{B'}_f\leq e^{4\varepsilon R^2}\omega_n R^n\leq e^{4\varepsilon}\omega_n R^n,
\]
so that
\[
0\leq \tau-R\leq R(e^{4\varepsilon/n}-1)\leq \frac{8\varepsilon}{n},
\]
for $\varepsilon$ small, so that
\[
0\leq\frac\varepsilon R-\frac\varepsilon\tau\leq\frac{8\varepsilon^2}{nR\tau}\leq\frac{8\varepsilon^2}{nR^2}=o(\varepsilon).
\]
By construction, $\psi=\varepsilon(r-\varepsilon^{-1})^2+o(\varepsilon)$ for $r\geq \varepsilon^{-1}$. This allows us to conclude that
\begin{equation}\label{eq:combine_eq_2}
    \E(B)-\E(B')\geq \varepsilon\Bigl(\frac{n-1}{R}\int_{B'}\hat x_1^2\,dx-\int_{\partial B'}\hat x_1^2\,d\Haus^{n-1}+2\int_{B'}\hat x_1\,dx\Bigr)+o(\varepsilon).
\end{equation}
It suffices now to prove that the quantity
\[
\frac{n-1}{R}\int_{B'}\hat x_1^2\,dx-\int_{\partial B'}\hat x_1^2\,d\Haus^{n-1}+2\int_{B'}\hat x_1\,dx,
\]
is strictly positive, independently of $\varepsilon$ and for any fixed $R<1$. This can be explicitly showed: let $(r,\varphi_1,\dots,\varphi_{n-1})\in\R_+\times (0,\pi)^{n-2}\times(0,2\pi)$ denote the usual hyperspherical coordinates of $\R^n$ centered in $B'$ and so that $\hat x_1=r\cos(\varphi_1)+R$. Then, since $r\cos(\varphi_1)$ has zero mean on $B'$ and $\partial B'$, we get that
\[
\int_{B'}\hat x_1\,dx=R^{n+1}\omega_n,
\]
and from $\hat x_1^2=r^2\cos^2(\varphi_1)+R^2+2Rr\cos(\varphi_1),$ we obtain
\begin{align*}
\int_{B'}\hat x_1^2\,dx&=\int_{S^{n-1}}\int_0^Rr^{n-1}(r^2\cos^2(\varphi_1)+R^2)\sin^{n-2}(\varphi_1)\dots \sin(\varphi_{n-2})\,dr\,d\varphi_1\dots\,d\varphi_{n-1}\\
&=\frac{R^{n+2}}{n+2}\int_{S^{n-1}}\cos^2(\varphi_1)\,d\Haus^{n-1}+R^{n+2}\omega_n,
\end{align*}
and similarly
\begin{align*}
\int_{\partial B'}\hat x_1^2\,d\Haus^{n-1}&=\int_{S^{n-1}}R^{n-1}(R^2\cos^2(\varphi_1)+R^2)\sin^{n-2}(\varphi_1)\dots \sin(\varphi_{n-2})\,d\varphi_1\dots\,d\varphi_{n-1}\\
&=R^{n+1}\int_{S^{n-1}}\cos^2(\varphi_1)\,d\Haus^{n-1}+nR^{n+1}\omega_n.
\end{align*}
Explicitly, integrating by parts we get that
\begin{align*}
\int_{S^{n-1}}\cos^2(\varphi_1)\,d\Haus^{n-1}&=\int_{S^{n-1}}\cos^2(\varphi_1)\sin^{n-2}(\varphi_1)\dots \sin(\varphi_{n-2})\,d\varphi_1\dots\,d\varphi_{n-1}\\
&=\frac{1}{n-1}\int_{S^{n-1}}\sin^{n}(\varphi_1)\dots \sin(\varphi_{n-2})\,d\varphi_1\dots\,d\varphi_{n-1}\\
&=\frac{1}{n}\int_{S^{n-1}}\sin^{n-2}(\varphi_1)\dots \sin(\varphi_{n-2})\,d\varphi_1\dots\,d\varphi_{n-1}\\
&=\omega_n.
\end{align*}
The proof is completed since
\begin{align*}
\frac{n-1}{R}&\int_{B'}\hat x_1^2\,dx-\int_{\partial B'}\hat x_1^2\,d\Haus^{n-1}+2\int_{B'}\hat x_1\,dx\\
&=\frac{n-1}{n+2}R^{n+1}\omega_n+(n-1)R^{n+1}\omega_n-(n+1)R^{n+1}\omega_n+2R^{n+1}\omega_n\\
&=\frac{n-1}{n+2}R^{n+1}\omega_n>0,
\end{align*}
for every $n\geq 2$.
\end{proof}

\section{Characterization of optimal sets}\label{sec:optimal-sets}

The proof Theorem \ref{thm:uniqueness}, which is an adaptation of the beautiful argument by Chambers \cite{chambers}, consists of two fundamental steps. In the first one, taking advantage of the symmetry of the weights, we decrease the class of competitors to the family of spherically symmetric domains. Thanks to the qualitative properties of optimal sets explained in Section \ref{sec:existence_variations}, we deduce key regularity properties of the curve $\gamma$ describing a spherically symmetrized set minimizing \eqref{eq:optimization_problem}, together with a particular second order ordinary differential equation solved by $\gamma$ inherited from the constancy of the weighted mean curvature in the sense of Definition \ref{def:weighted_mean_curvature}. The second step consists of a delicate analysis of the profile curve $\gamma$. Arguing by contradiction, we decompose it into three fundamental subsequent curves: an upper curve, a lower curve, and a curl curve. The existence and finiteness of the upper curve follow from choosing the starting point of $\gamma$ to be the furthest point from the origin combined with the boundedness of optimal sets. A comparison argument shows that the lower curve has to curve faster than the upper curve. As a consequence the lower curve cannot cross the axis of symmetrization, otherwise, the angle of incidence would be greater than $\frac\pi2$, contradicting the optimality of $E$. Finally, the curl curve behavior follows from mean convexity, leading to a contradiction since a profile $\gamma$ that curls cannot describe an embedded, spherically symmetrized set.

Remarkably, this whole argument only takes advantage of the \emph{sationarity} and \emph{qualitative} properties of optimal sets to conclude a full characterization of minimizers of \eqref{eq:optimization_problem}.

\subsection{Spherical symmetrization and terminology}
Fix now and for all $\psi,g$ to be monotonically increasing, strictly admissible weights of class $C^3$ in the sense of Definition \ref{def:admissible_weights}. For an arbitrary volume $v>0$, let $E$ be a minimizer of \eqref{eq:optimization_problem} of volume $\abs{E}_f=v$, whose existence and boundedness is ensured by Theorem \ref{thm:existence}. Moreover, recall that Theorem \ref{thm:regularity} implies that the boundary of $E$ is of class $C^{3,\alpha}$, for all $\alpha\in(0,1)$, away from a singular set of Hausdorff dimension $n-8$. We introduce now the concept of spherical symmetrization.
\begin{definition}\label{def:spherical_symm}
    Let $F\subset\R^n$ be a bounded set of finite perimeter, and $\xi\in S^{n-1}$.  Define $\tau(r):=\Haus^{n-1}(F\cap\partial B_r)$ and let $\eta:[0,\pi]\to [0,n\omega_n]$ be the function associating to every $s\in(0,\pi)$ the area of the spherical cap in $S^{n-1}$ of spherical radius $s$. We call
    \begin{equation}\label{eq:spherical_symm}
        F^\star_\xi:=\Bigl\{x\in\R^n\setminus\{0\}: \arccos\Bigl(\frac{x\cdot \xi}{\abs{x}} \Bigr)< \eta^{-1}(\abs{x}^{1-n}\tau(\abs{x}))\Bigr\},
    \end{equation}
    the \emph{spherical symmetrization of $F$ in the direction $\xi$}.
\end{definition}
\begin{proposition}\label{prop:spherical_symm} For every $\xi\in S^{n-1}$ the spherical symmetrization in the direction $\xi$ of every bounded set of finite perimeter $F\subset\R^n$ is a well defined set of finite perimeter. Moreover, for  every $f,g$ positive Borel functions radially symmetric with respect to the the origin
\[
\abs{F^\star_\xi}_f=\abs{F}_f,\quad \G(F^\star_\xi)=\G(F),\quad\Pe(F^\star_\xi)\leq\Pe(F_\xi).
\]
In particular, $\E(F^\star_\xi)\leq \E(F)$.
\end{proposition}
\begin{proof}
    Spherical symmetrization is a well-known rearrangement technique whose properties derive essentially from the proof of the classical Steiner symmetrization by planes \cite{steiner1838einfache,kawohl1986isoperimetric}. We refer to Section 9.2 in~\cite{burago2013geometric},
    Remark 4 in~\cite{morgan2011steiner}, and
    Proposition 6.2 in~\cite{morgan-pratelli2013}.
\end{proof}
Take an arbitrary orthonormal basis of $\{e_1,\dots,e_n\}$ of $\R^n$, and set $\xi=e_1$. Thanks to Proposition \ref{prop:spherical_symm}, $E^\star_\xi$ is also a minimizer of \eqref{eq:optimization_problem}, and hence enjoys the same regularity properties of Theorem \ref{thm:regularity}. Set
\[
R^\star:=\ess\sup\Bigl\{\abs{x}\in E^\star_\xi\Bigr\},
\]
and $E_K\subset E^\star_\xi$ to be the connected component containing $R^\star e_1$ in its closure. Let the continuous curve
\begin{equation}\label{eq:def_gamma}
\gamma:[-\beta,\beta]\to \spann\{e_1,e_2\}=\R^2,
\end{equation}
be the arc-length parametrization of $\partial\{E_K\cap\spann\{e_1,e_2\}\}$ so that $\gamma(0)=(R^\star,0)$. Before stating the properties of $\gamma$, we need to introduce three objects that will be crucial in the comparison arguments we will perform in the second step of the proof on Theorem \ref{thm:uniqueness}.
\begin{definition}\label{def:comparison_circles}
    For $s\in(-\beta,\beta)$ we define $C_s$ to be the oriented (possibly degenerated) circle tangent to $\gamma(s)$ with center on $e_1$ with coordinate $F(s)$, radius $R(s)$, and curvature $\lambda(s)$. We define $A_s$ to be the oriented osculating (possibly degenerated) circle at $\gamma(s)$, with $\alpha(t)$ its arc-length parametrization with normal $n(t)$ and $\alpha(\tilde s)=\gamma(s)$. Finally, we let $\tilde C_t$ be the oriented (possibly degenerated) circle tangent to $\alpha(t)$ with center on $e_1$ of coordinate $\tilde F(t)$, radius $\tilde R(t)$, and curvature $\tilde\lambda(s)$.
\end{definition}
Let $\nu(s):=\dot\gamma^\perp(s)=(-\dot\gamma_2(s),\dot\gamma_1(s))$ and $\kappa(s)$ the normal and the curvature of $\gamma$ at $s\in(-\beta,\beta)$, respectively. Then the mean curvature of $\gamma$ can be explicitly computed as follows:
\begin{lemma}\label{lem:regularity_gamma}The curve $\gamma=(\gamma_1,\gamma_2)$ is a simple closed curve of class $C^{3,\alpha}$ in $(-\beta,\beta)$, and it is axially symmetric, that is $\gamma_2(s)=-\gamma_2(-s)$ for every $s\in[0,\beta)$. Moreover, $\gamma$ solves
\begin{equation}\label{eq:ODE_gamma}
    \Ha_\gamma:=\kappa(t)+(n-2)\lambda(t)+\psi'(\abs{\gamma(t)})\frac{\gamma(t)\cdot\nu(t)}{\abs{\gamma(t)}}+g(\abs{\gamma(t)})=\text{constant},
\end{equation}
in $(-\beta,\beta)$.
\end{lemma}
\begin{proof}
    The axial symmetry of $\gamma$ follows directly from Definition \ref{def:spherical_symm}. Theorem \ref{thm:regularity} and Proposition \ref{prop:spherical_symm} imply that the boundary of $E^\star_\xi$ is of class $C^{3,\alpha}$ outside a singular set of Hausdorff dimension at most $n-8$. By symmetry, a singularity of $\gamma$ in $(0,\beta)$ implies a singularity of dimension at least $n-2$ on $\partial E^\star_\xi$, which is a contradiction. Moreover, thanks to the regularity properties of points whose tangent cone lies in a half-space, we infer that $\gamma$ is also smooth at the origin (we refer to \cite{morgan-regularity} for details). Equation \eqref{eq:ODE_gamma} is nothing else that the weighted mean curvature $\Ha$ of $\partial E^\star_\xi$ expressed in terms of $\gamma$, which we know to be constant in virtue of Lemma \ref{lem:first_variations}.
\end{proof}
For convenience, we will split Equation \eqref{eq:ODE_gamma} in two parts $\Ha_\gamma=H_0+H_1$, where
\begin{equation}\label{eq:def_H_0_H_1}
    H_0:=\kappa+(n-2)\lambda,\qquad H_1:=\gamma'(\abs{\gamma})\frac{\gamma\cdot\nu}{\abs{\gamma}}+g(\abs{\gamma}),
\end{equation}
so that $H_0$ represents the unweighted mean curvature of $K$. Furthermore, we set
\begin{equation}\label{eq:def_tilde_H_0}
    \tilde H_1(t):=\gamma'(\abs{\alpha})\frac{\alpha\cdot n}{\abs{\alpha}}+g(\abs{\alpha}).
\end{equation}
The following elementary lemma clarifies why Definition \ref{def:comparison_circles} is appropriate to perform comparison arguments, in the sense that $C_s$, $A_s$, and $\tilde C_t$ are a good approximation of the mean curvature $\Ha_\gamma$ at $s$.
\begin{lemma}\label{lem:good_approx_by_circles}Let $s\in(-\beta,\beta)$. Then, the following identities hold at $s$: $\dot\alpha(\tilde s)=\dot\gamma(s)$, $\ddot\alpha(\tilde s)=\ddot\gamma(s)$, $\tilde F'(\tilde s)=F'(s)$, $\tilde R'(\tilde s)=R'(s)$, and $\tilde\lambda'(\tilde s)=\lambda'(s)$.
\end{lemma}
\begin{proof}This Lemma follows from Definition \ref{def:comparison_circles}, since $C_s$ and $A_s$ approximate $\gamma$ in $s$ up to second and third order, respectively.
\end{proof}

We conclude this preliminary part by proving a lemma stating the two main ingredients to obtain contradictions in the proof of Theorem \ref{thm:uniqueness}.
\begin{lemma}\label{lem:tangent restrict}
One has that
\begin{equation}\label{eq:tangent_restriction}
    \gamma(s) \cdot \dot\gamma(s) \leq 0,\qquad\forall s\in(0,\beta),
\end{equation}
and if $\lim_{s\to\beta^+}\dot\gamma(s)$ exists, then
\begin{equation}\label{eq:no_apple_points}
    \lim_{s\to\beta^+}\dot\gamma(s)\cdot e_1\geq 0.
\end{equation}
\end{lemma}
\begin{proof}
    Equation \eqref{eq:tangent_restriction} follows from the observation that $\abs{\gamma}$ is non-increasing in $t>0$, see Lemma 2.6 in~\cite{chambers}. Equation \eqref{eq:no_apple_points} is in fact a consequence of the monotonicity of the isoperimetric profile of Proposition \ref{prop:increasing_profile}. By contradiction, suppose that the tangent cone to $E_K$ at $y=\lim_{s\to\beta^+}\gamma_1(s)>0$ has internal angle $2\theta<\pi/2$. Let $\delta>0$ be small enough and define $\hat K$ to be the set obtained by truncating vertically $\gamma$ at $x_1=\abs{y}-\delta$, that is
    \[
    \hat E_K=E_K\cup\{(\gamma_1(s),x_2,\dots,x_n):\abs{y}-\delta\leq\gamma_1(s)\leq\abs{y}\text{ and }\abs{(x_2,\dots,x_n)}<\abs{\gamma_2(s)}\}.
    \]
    Since $f>0$, $\abs{\hat E_K}_f>\abs{E_K}_f$. Arguing like in Proposition \ref{prop:psi_monotone} to bound the approximations errors replacing $r$ with the $e_1$ coordinate, one can estimate
    \begin{align*}
    \Pe(E_K)-\Pe(\hat E_K)\geq \omega_ne^{\psi(\abs{y})}\tan(\theta)^{n-1}\delta^{n-1}\Bigl(\sqrt{1+\tan(\theta)^{-2}}-1\Bigr)+O(\delta^n),
    \end{align*}
    and for some constant $c=c(n)>0$
    \[
    \G(E_K)-\G(\hat E_K)\geq -c(n)e^{\psi(\abs{y})}\tan(\theta)^{n-1}\delta^{n}+o(\delta^{n})=O(\delta^n).
    \]
    Hence, for  $\delta$ small enough, we conclude that
    \[
    \E(E_\xi^\star\cup\hat E_K)<\E(E_\xi^\star),\qquad \abs{E_\xi^\star\cup \hat E_K}_f>\abs{E_\xi^\star}_f,
    \]
    which contradicts Proposition \ref{prop:increasing_profile} by minimality of $E_\xi^\star$. 
\end{proof}
\subsection{Some useful lemmas}
In this section, we investigate even further the qualitative properties of $\gamma$. We first show that all region solving \eqref{eq:optimization_problem} (so $E^\star_\xi$ in particular) is mean convex. This was proven in the log-convex setting by Morgan and Patelli in \cite{morgan-pratelli2013}.
\begin{lemma}\label{lem:mean-convexity}The minimizing set $E$ is mean convex, that is $H_0\geq 0$.
\end{lemma}
\begin{proof}
Let $z$ be the furthest point from the origin in the set $\partial E.$ Note that this is a regular point. Then clearly we have $H_0(z)\geq 0.$ Let $x\in \partial E$ be such that $|x|\leq |z|$. Then
\begin{equation}\label{eq:meanconvexity}
\begin{split}
H_1(z)&=\nabla\psi(z)\cdot\nu(z)+g(\abs{z})=\frac{\psi'(\abs{z})}{\abs{z}}z\cdot\nu(z)+g(\abs{z})\\
&=\psi'(\abs{z})+g (\abs{z})\geq \psi'(\abs{x})+g (\abs{x})\geq \frac{\psi'(\abs{x})}{\abs{x}}x\cdot\nu(x)+g(\abs{x})\\
&=\nabla\psi(x)\cdot\nu(x)+g(\abs{x})=H_1(x),
\end{split}
\end{equation}
 and since 
\begin{align}
    H_0(x) + H_1(x) = H_0(z) + H_1(z)
\end{align}
we get that $H_0(x)=H_1(z)-H_1(x)+H_0(z)\geq H_0(z)\geq 0$, as desired.
\end{proof}
In the next lemma, we rule out a particular initial condition for which we know already that $\gamma$ has to be a centered circle. We refer to Lemma 3.2 in~\cite{chambers}.
\begin{lemma}\label{lem:ODE_uniqueness} If for some $s\in[0,\beta)$, $F(s)=0$ and $\kappa(s)=\lambda(s)$, then $\gamma=C_s$.
\end{lemma}
\begin{proof}
This follows from the fact that the curves $C_s$ and $\gamma$ satisfy the same ODE, $\mathbf{H} = c$ and agree at the same point $s$. Therefore the conclusion follows by the uniqueness of solutions of ODEs.
\end{proof}
From now on, we suppose that $\gamma$ does not meet the requirements of Lemma \ref{lem:ODE_uniqueness}, since otherwise we would be done. We thus argue by contradiction and begin by analyzing the behavior of the curve at its initial time.
\begin{lemma}\label{lem:prop-kappa(0)}
At the initial time $t=0$ the following holds
\begin{enumerate}
    \item $\kappa(0)=\lambda(0)>0$, $\dot\kappa(0)=0$.
    \item If $F(0)>0$, then $\ddot \kappa(0)>0$.
\end{enumerate}
\end{lemma}
\begin{proof}
$\mathrm{(1)}$ See proof of Lemma 4.8 in \cite{BoyerBrownChambersLovingTammen+2016}. The idea is to observe that $\kappa(0)=\lambda(0)$ is the same as showing that $F(0)=1-1 / \kappa(0).$ Then using the fact that $\kappa(0)>0$ and $\kappa$ is continuous there exists a neighborhood around $s=0$ where $\dot{\gamma}_1(s)\neq 0$ and therefore we have
\begin{align}
F(0)=\lim _{s \rightarrow 0} F(s)=\lim _{s \rightarrow 0} \frac{\gamma(s) \cdot \gamma^{\prime}(s)}{\gamma_1{ }^{\prime}(s)}=1-\frac{1}{\kappa(0)}.
\end{align}
$\mathrm{(2)}$ Following the argument in Lemma 6.3 in~\cite{BoyerBrownChambersLovingTammen+2016} it is suffices to prove $\tilde H_1''(0)<0$. Let $\alpha(t)=(a+R\cos(t/R),R\sin(t/R))$ be a parametrization of $A_0$, with $a=\tilde F(0)=F(0)>0$ and $R=\tilde R(0)=R(0)$. Then
\begin{align}\label{eqn:H_1 derivative}
\tilde H_1'=\frac{d}{dr}\Big\vert_{r=\abs{\alpha}}\Bigl(\frac{\psi'(r)}{r}\Bigr)(\dot\alpha\cdot\alpha)\Bigl(\frac{\alpha\cdot n}{\abs{\alpha}}\Bigr)+\frac{\psi'(\abs{\alpha})}{\abs{\alpha}}(\dot\alpha\cdot n+\alpha\cdot\dot n)+\frac{d}{dr}\Big\vert_{r=\abs{\alpha}}(g(r)/f(r))\frac{\dot\alpha\cdot\alpha}{\abs{\alpha}}.    
\end{align}
Noticing that at $\dot\alpha(0)\cdot\alpha(0)=0$, we obtain
\begin{align*}
\tilde H_1''(0)&=\frac{d}{dr}\Big\vert_{r=\abs{\alpha}}\Bigl(\frac{\psi'(r)}{r}\Bigr)(\ddot\alpha\cdot\alpha+\abs{\dot\alpha}^2)\Bigl(\frac{\alpha\cdot n}{\abs{\alpha}}\Bigr)+\frac{\psi'(\abs{\alpha})}{\abs{\alpha}}(\ddot\alpha\cdot n+2\dot\alpha\cdot\dot n+\alpha\cdot\ddot n)\\
&\quad +\frac{d}{dr}\Big\vert_{r=\abs{\alpha}}(g(r)/f(r))\frac{\ddot\alpha\cdot\alpha+\abs{\dot\alpha}^2}{\abs{\alpha}}.
\end{align*}
Since at $t=0$ we have $\ddot\alpha\cdot\alpha=-\abs{\alpha}/R$, $\alpha\cdot n=\abs{\alpha}$, $\ddot\alpha\cdot n=-1/R$, $\dot\alpha\cdot \dot n=1/R$, and $\alpha\cdot\ddot n=-\abs{\alpha}/R^2$ and $\abs{\alpha(0)}>R$ we get that
\[ 
\tilde H_1''(0)=\underbrace{\Bigl(\psi''+g'\Bigr)}_{>0}\frac{1}{\abs{\alpha}}\underbrace{(1-\abs{\alpha}/R)}_{<0}+\underbrace{\psi'}_{\geq 0}\frac{1}{\abs{\alpha}}\underbrace{(1-\abs{\alpha}/R)}_{<0}\underbrace{(1/R-1/\abs{\alpha})}_{>0}<0.
\]
\end{proof}
\subsection{Upper Curve}
In this section, we determine the initial behavior of $\gamma$, which is defined as follows.
\begin{definition}The upper curve $K$ is the set of all $s\in[0,\beta)$ such that 
\begin{enumerate}
    \item $\dot\gamma(s)$ lies in the II quadrant
    \item $\kappa(s)\geq\lambda(s)>0$.
\end{enumerate}
\end{definition}
We now record some elementary properties of the upper curve.
\begin{lemma} \label{lem:prop-upper-curve} 
Let $s\in K$ be in the upper curve and $\gamma$ is not a centered circle. Then the following properties hold
\begin{enumerate}
    \item $K\neq\varnothing$ and $\delta:=\sup K>0$.
    \item $F(0)\geq 0$.
    \item $F'(s)\geq 0$.
    \item $F(s)\geq R(s)$.
    \item If $\kappa(s)=\lambda(s)$, then $\lambda'(s)=0$ but $\kappa'(s)>0$.
\end{enumerate}  
\end{lemma}
\begin{proof}
\begin{enumerate}
    \item The first part follows from a continuity argument. For more details see for instance Lemma 6.5 in~\cite{BoyerBrownChambersLovingTammen+2016}.
    \item If $F(0)<0$, then there are points in $[0, \beta)$ near 0 on $\gamma$ that lie outside the centered circle of radius $\gamma(0)$. This contradicts the fact that $|\gamma(s)|\leq |\gamma(0)|$ for all $s\in [0,\beta)$. Thus $F(0) \geq 0$.
    \item To show this, it suffices to prove that $\tilde F'(\tilde s)\geq 0$, again by order of approximation. This result is purely geometrical and independent of our particular ODE. See Lemma 5.3 in the Appendix of~\cite{chambers}.
    \item This fact too is geometric in nature and independent of our ODE. For more details, see Lemma 6.6 in~\cite{BoyerBrownChambersLovingTammen+2016}.
    \item Following the proof of Lemma 6.7 \cite{BoyerBrownChambersLovingTammen+2016} it suffices to prove $\tilde H_1'(\tilde s)<0$. Using arc length parametrization of $A_s$ we first compute
\begin{align*}
\alpha(t) &=\left(a+r \cos \left(\frac{t}{r}\right), b+r \sin \left(\frac{t}{r}\right)\right)\\
\dot{\alpha}(t) &=(-\sin (t / r), \cos (t / r))\\
n(t) & =(\cos (t / r), \sin (t / r)) \\ 
\dot{n}(t) & =\frac{(-\sin (t / r), \cos (t / r))}{r}=\frac{1}{r}\dot \alpha(t)\\
\end{align*}
Since $A_s=C_s$, we have that $b=0, a=F(s)$, and $r=R(s).$
Let $\theta\in[-\pi,\pi]$ be the angle between $\alpha$ and $\dot\alpha$. Then, we have that
\begin{align*}
\tilde H_1'&=\frac{d}{dt}\Bigl(\psi'(\abs{\alpha})\sin(\theta)+g(\abs{\alpha})\Bigr)\\
&=\psi''(\abs{\alpha})\cos(\theta)\sin(\theta)+\psi'(\abs{\alpha})(\sin(\theta))'+g'(\abs{\alpha})\cos(\theta)\\
&=(\psi''+g')\cos(\theta)\sin(\theta)+\psi'(\sin(\theta))'+g'\cos(\theta)(1-\sin(\theta)).
\end{align*}
Notice that $\theta(\tilde s)\in(\pi/2,\pi)$ and $\abs{\alpha(\tilde s)}\geq R$ (by triangle inequality). It follows that in $t=\tilde s$, $\cos(\theta)\in(-1,0)$, $\sin(\theta)\in(0,1)$, and
\[
\tilde H'<\psi'(\sin(\theta))'.
\]
We are left to prove  $(\sin(\theta))'\leq 0$. Now, a direct computation shows that
\begin{align*}
(\sin(\theta))'&=\Bigl(\frac{\alpha\cdot n}{\abs{\alpha}}\Bigr)'=\frac{\alpha\cdot\dot n+\dot\alpha\cdot n}{\abs{\alpha}}-\frac{(\alpha\cdot\dot\alpha)(\alpha\cdot n)}{\abs{\alpha}^3}\\
&=\frac{\alpha\cdot\dot n}{\abs{\alpha}}-\frac{\sin(\theta)\cos(\theta)}{\abs{\alpha}}=\frac{\alpha\cdot\dot \alpha}{R\abs{\alpha}}-\frac{\sin(\theta)\cos(\theta)}{\abs{\alpha}}\\
&=\underbrace{\cos(\theta)}_{<0}\Bigl(\frac{1}{R}-\frac{\sin(\theta)}{\abs{\alpha}}\Bigr)<\cos(\theta)\Bigl(\frac 1R-\frac{1}{\abs{\alpha}}\Bigr)\leq 0,
\end{align*}
as desired.
\end{enumerate}
\end{proof}
Next, we observe that the upper curve cannot extend indefinitely and its endpoint must satisfy certain geometric properties.
\begin{lemma}\label{lem:prop delta} $K$ is not empty, and so $\delta$ exists. Furthermore, if $\gamma$ is not a centered circle, then $\delta$ has the following properties:
\begin{enumerate}
\item  $\delta<\beta$.
\item $\delta \in K$.
\item $\gamma_1(\delta)\geq F(s)$ for any $s\in [0,\delta].$
\item $\gamma_1(\delta)>0$.
\item $\gamma^{\prime}(\delta)=(-1,0)$.
\item $\dot{\kappa}(\delta)>0$
\end{enumerate}
\end{lemma}
\begin{proof}
See proof of Lemma 6.8 in~\cite{BoyerBrownChambersLovingTammen+2016} for the proof of properties (1) to (5). To see property (6) it suffices to show that $\Tilde{H}_1'(\pi \Tilde{R}(\Tilde{s})/2) < 0.$ To this end recall
\begin{align*}
\tilde H_1'=(\psi''+g')\cos(\theta)\sin(\theta)+\psi'(\sin(\theta))'+g'\cos(\theta)(1-\sin(\theta))
\end{align*}
where $\theta$ is the angle between $\alpha$ and $\dot{\alpha}.$ To deduce a strict inequality we simply observe that $(\sin (\theta))'<0$ since the other terms are non-positive since $\cos(\theta)\in (-1,0),\sin(\theta)\in (0,1)$, $\psi''+g'\geq 0$, $\psi'\geq 0$ and $g'\geq 0$.  Observe that since $\gamma'(\delta)=(-1,0)$, the radius $R=R(\delta)$ of the canonical circle $C_\delta$ satisfies $R=r+b$ and $a=F(\delta)$. Therefore from part (3) of Lemma \ref{lem:prop-upper-curve} we have $a>R>0$ which in turn implies that
\begin{align}
|\alpha(\Tilde{s})|^2 = a^2 + (b+r)^2 = a^2 + R^2 > R^2.
\end{align}
Thus $|\alpha(\Tilde{s})|>R$ which implies $(\sin(\theta))'<0$. Thus $H_1'(\pi\Tilde{R}(\Tilde{s})/2)<0$ as desired.
\end{proof}

\subsection{Lower Curve}
Since we now know that the upper curve exists, is finite, and ends horizontally, we describe the subsequent behavior of the curve $\gamma$ as follows.
\begin{definition}\label{defn:lower-curve}
Define the lower curve $L$ as the set of all points $s$ in $[\delta, \beta)$ such that for all $t \in[\delta, s]$ the following hold:
\begin{enumerate}
    \item $\gamma^{\prime}(t)$ is in the III quadrant with $\gamma^{\prime}(t) \neq(-1,0)$ if $t>\delta$,
    \item If $\bar{t}$ is the unique point in $K$ with $\gamma_2(\bar{t})=\gamma_2(t)$, then $\kappa(\bar{t}) \leq \kappa(t)$.
\end{enumerate}
\end{definition}
\begin{remark}[The lower curve $L$ is non-empty]
As $\gamma^{\prime}(\delta)=(-1,0), \kappa(\delta)>0$, and $\kappa^{\prime}(\delta)>0$ by continuity one sees that the conditions in Definition \ref{defn:lower-curve}  hold on an interval $[\delta, \delta+\varepsilon)$. Thus, $L$ is nonempty and has a supremum, $\eta := \sup L.$
\end{remark}
The key point in the analysis of the behavior of the lower curve is to prove that it curves faster than the upper curve, so to obtain a contradiction via Equation \eqref{eq:no_apple_points} in the case of intersection with the $e_1$ axis. We start with two useful definitions.
\begin{definition}
Define $h:\left(\gamma_2(\eta), \gamma_2(\delta)\right) \rightarrow(0, \delta)$ as the unique time $t \in(0, \delta)$ such that $\gamma_2(t)=y$. Similarly, we define $k:\left(\gamma_2(\eta), \gamma_2(\delta)\right) \rightarrow(\delta, \eta)$ by letting $k(y)$ be the unique $t \in(\delta, \eta)$ such that $\gamma_2(t)=y$.    
\end{definition}
\begin{definition}
Given an $e_2$-coordinate with $y \in\left(\gamma_2(\eta), \gamma_2(\delta)\right)$, let
\begin{align}
p(y) &=2 \gamma_1(\delta)-\gamma_1(h(y))\\
q(y) &=\gamma_1(k(y)).    
\end{align}  
The function $q$ gives the $e_1$-coordinate of a point in $\gamma(L)$ with a given $e_2$-coordinate. If we begin with the point in $\gamma(K)$ with a given $e_2$-coordinate, then $p$ gives the $e_1$-coordinate of the reflection of this point over the line $x=\delta$. 
\end{definition}
The next two lemmas show that the lower curve curves faster than the upper one. 
\begin{lemma}\label{lem:e1 and angle comparison}
For each $s \in[\delta, \eta)$, let $\bar{s}$ be the unique point in $K$ so that $\gamma_2(\bar{s})=\gamma_2(s)$. Then the following hold:
$$
\begin{gathered}
\gamma_1(\bar{s})-\gamma_1(\delta) \geq \gamma_1(\delta)-\gamma_1(s), \\
\theta\left(\gamma^{\prime}(s)\right) \geq 2 \pi-\theta\left(\gamma^{\prime}(\bar{s})\right) .
\end{gathered}
$$    
\end{lemma}
\begin{proof}
The functions $p$ and $q$ satisfy the assumptions of Proposition 3.8 in \cite{chambers}, by the definition of Lower Curve $L$. Thus we get $p\leq q$ which implies the first inequality while $\theta(t_p)\geq \theta(t_q)$ implies the second inequality. For more details, see proof Lemma 6.13 in \cite{BoyerBrownChambersLovingTammen+2016}.
\end{proof}
\begin{lemma}\label{lem:lower curve curves more} Let $s \in(\delta, \eta)$, and suppose that $\gamma^{\prime}(s) \neq(0,-1)$. If $\bar{s}$ is the unique point in $K$ so that $\gamma_2(\bar{s})=\gamma_2(s)$, then $\kappa(s)>\kappa(\bar{s})$.    
\end{lemma}
\begin{proof}
Since $H_0+H_1$ is constant we have,
$$
\kappa(s)+(n-2) \lambda(s)+H_1(s)=\kappa(\bar{s})+(n-2) \lambda(\bar{s})+H_1(\bar{s}) .
$$
It can be shown using right triangle trigonometry and the second inequality from Lemma \ref{lem:e1 and angle comparison} that $\lambda(s) \leq \lambda(\bar{s})$. Thus, to prove that $\kappa(s)>\kappa(\bar{s})$, it suffices to prove that $H_1(s)<H_1(\bar{s})$. Recall that 
\begin{align}
    H_1(s) = \psi'(|\gamma(s)|) \frac{\gamma(s)\cdot \nu(s)}{|\gamma(s)|} + g (|\gamma(s)|).
\end{align}
We show that $(v_1,v_2) = (\gamma^{\prime}(\bar{s}),\gamma^{\prime}(s))$ are admissible with respect to $\gamma(\bar{s})$ and $\gamma(s)$ and then apply Proposition 3.9 in \cite{chambers}. To this end, we check
\begin{enumerate}
    \item $F(\bar{s}) \geq F(0) >0$ by property (2) in Lemma \ref{lem:prop-upper-curve}.
    \item $R(s) \geq R(\bar{s})$ holds since $\lambda(s)\leq \lambda(\bar{s})$.
    \item $\gamma_1(\bar{s})-F(\bar{s}) \geq \gamma_1(\bar{s})-\gamma_1(\delta) \geq \gamma_1(\delta)-\gamma_1(s) \geq F(\bar{s})-\gamma_1(s)$ where we used property (3) in Lemma \ref{lem:prop delta} to get $\gamma_1(\delta)\geq F(\bar{s})$ and the second inequality in Lemma \ref{lem:e1 and angle comparison}.
\end{enumerate}
Since $\gamma^{\prime}(s)$ is not equal to $(0,-1), \gamma^{\prime}(s)$ lies strictly in the third quadrant we get strict inequalities from Proposition 3.9 in \cite{chambers}
\begin{align}
    |\gamma(\bar{s})| &> |\gamma(s)|\\
    \frac{\gamma(\bar{s})\cdot \nu(\bar{s})}{|\gamma(\bar{s})|} &> \frac{\gamma(s)\cdot \nu(s)}{|\gamma(s)|}.
\end{align}
Thus,
\begin{align*}
    H_1(\bar s)&=\psi'(\abs{\gamma(\bar s)})\frac{\gamma(\bar s)\cdot\nu(\bar s)}{\abs{\gamma(\bar s)}^2}+g(\abs{\gamma(\bar s)})\\
    &=\Bigl(\psi'(\abs{\gamma(\bar s)})+g(\abs{\gamma(\bar s)})\Bigr)\frac{\gamma(\bar s)\cdot\nu(\bar s)}{\abs{\gamma(\bar s)}^2}+g(\abs{\gamma(\bar s)})\Bigl(1-\frac{\gamma(\bar s)\cdot\nu(\bar s)}{\abs{\gamma(\bar s)}^2}\Bigr)\\
    &>\Bigl(\psi'(\abs{\gamma(s)})+g(\abs{\gamma(s)})\Bigr)\frac{\gamma(\bar s)\cdot\nu(\bar s)}{\abs{\gamma(\bar s)}^2}+g(\abs{\gamma(\bar s)})\Bigl(1-\frac{\gamma(\bar s)\cdot\nu(\bar s)}{\abs{\gamma(\bar s)}^2}\Bigr)\\
    &=\psi'(\abs{\gamma(s)})\frac{\gamma(\bar s)\cdot\nu(\bar s)}{\abs{\gamma(\bar s)}^2}+g(\abs{\gamma(s)})+\Bigl(g(\abs{\gamma(\bar s)})-g(\abs{\gamma(s)})\Bigr)\Bigl(1-\frac{\gamma(\bar s)\cdot\nu(\bar s)}{\abs{\gamma(\bar s)}^2}\Bigr)\\
    &>H_1(s),
\end{align*}
since $\nu(\bar s)$ and $\gamma(\bar s)$ are in the second quadrant.
\end{proof}
We can now prove that the lower curve ends before crossing the axis $e_1$, completing the second (and most technical step) in the proof of Theorem \ref{thm:uniqueness}.
\begin{lemma}\label{lem:lower curve ends early}
If $\gamma$ is not a centered circle, then $\eta<\beta$ and $\eta \in L$.
\end{lemma}
\begin{proof} 
First by using Proposition 3.8 in \cite{chambers} and \ref{lem:lower curve curves more} one can deduce that $\eta = \beta$ implies the existence of point $z\in L$, such that $\gamma'(z)$ does not lie in third quadrant which is a contradiction. Thus $\eta < \beta$. To see why $\eta \in L$ we simply observe that $\gamma$ is smooth at $\eta$, $\gamma'(\eta)$ is in the third quadrant and $\kappa(\eta)\geq \kappa(\bar{\eta}).$ For more details, see proof of Lemma 3.14 in \cite{chambers}. 
\end{proof}
\begin{lemma}\label{lem:gamma_1(eta)}
We have that $\gamma_1(\eta)>0$.    
\end{lemma}
\begin{proof}
By the second inequality in Lemma \ref{lem:e1 and angle comparison} we have 
\begin{align}
\gamma_1(\bar{\eta})-\gamma_1(\delta) \geq \gamma_1(\delta)-\gamma_1(\eta)
\end{align}
Furthermore, $\gamma_1(\delta)=F(\delta) \geq F(\bar{\eta})$ and therefore
\begin{align}
\gamma_1(\bar{\eta})-F(\bar{\eta}) \geq \gamma_1(\bar{\eta})-\gamma_1(\delta) \geq \gamma_1(\delta)-\gamma_1(\eta) \geq F(\bar{\eta})-\gamma_1(\eta).
\end{align}
Finally, $\gamma_1(\bar{\eta})-F(\bar{\eta})<R(\bar{\eta})$, because $R(\bar{\eta})$ is the distance from $(F(\bar{\eta}), 0)$ to $\gamma(\eta)$. It follows that $\gamma_1(\eta) \geq$ $F(\bar{\eta})-\left(\gamma_1(\bar{\eta})-F(\bar{\eta})\right)>F(\bar{\eta})-R(\bar{\eta})>0$ by property (3) in Lemma \ref{lem:prop-upper-curve}. 
\end{proof}
In the following two lemmas, we prove that $\gamma$ curls after the end of the lower curve.
\begin{lemma}\label{lem:IV quad positive curvature}
 Let $s \in(0, \beta)$. If $\gamma_1(s) \geq 0$ and $\gamma^{\prime}(s)$ is in the fourth quadrant, then $\kappa(s)>0$.   
\end{lemma}
\begin{proof}
Since $\gamma^{\prime}(s)$ is in the fourth quadrant and $\gamma_2(s)>0, \lambda(s) \leq 0$. Since $\gamma(s)$ is in the first quadrant and $v(s)$ is in the third, $\gamma(s) \cdot v(s) \leq 0$. 
\[
\kappa(s)+(n-2)\lambda(s)+H_1(s)=\kappa(0)+(n-2)\lambda(0)+H_1(0),
\]
and $\kappa(0)=\lambda(0)>0$ we get
\[
\kappa(s)=-(n-2)\lambda(s)-H_1(s)+(n-1)\kappa(0)+H_1(0)>H_1(0)-H_1(s)\geq 0,
\]
thanks to Equation \eqref{eq:meanconvexity}.
\end{proof}
\begin{lemma}\label{lem:after lower curve}
 For $s \in(\eta, \beta), \gamma^{\prime}(s)$ lies strictly in the fourth quadrant, and $\kappa(s)>0$.   
\end{lemma}
\begin{proof}
We follow the argument in Lemma 6.17 in \cite{BoyerBrownChambersLovingTammen+2016}. Define $A \subset(\eta, \beta)$ so that $s \in A$ if and only if for all $t \in(\eta, s), \gamma^{\prime}(t)$ lies strictly in the fourth quadrant and $\kappa(t)>0$. Note that $A$ is nonempty because $\kappa(\eta)>0, \gamma^{\prime}(\eta)=(0,-1)$, and $\kappa$ is continuous at $\eta$. Thus, $A$ has a supremum $\omega$. To prove the lemma we show that $\omega=\beta$.

Suppose for contradiction that $\omega<\beta$. Then $\gamma$ is smooth at $\omega$; in particular, $\gamma^{\prime}(\omega)$ and $\kappa(\omega)$ are defined. Since $\gamma^{\prime}(t)$ lies in the fourth quadrant for all $t \in(\eta, \omega), \gamma^{\prime}(\omega)$ is in the fourth quadrant. Since $\kappa>0$ on $(\eta, \omega), \gamma^{\prime}(\omega)$ is not equal to $(0,-1)$. Furthermore, $\gamma_1(\omega)>0$, as $\gamma_1(\eta)>0$ from Lemma \ref{lem:gamma_1(eta)} and $\gamma^{\prime}$ lies in the fourth quadrant on $(\eta, \omega)$. If $\gamma^{\prime}(\omega)$ were equal to $(1,0)$, then we would have $\gamma^{\prime}(\omega) \cdot \gamma(\omega)=$ $\gamma_1(\omega)>0$, contradicting the Tangent Restriction Lemma \ref{lem:tangent restrict}. Thus $\gamma^{\prime}(\omega)$ lies strictly in the fourth quadrant. By Lemma \ref{lem:IV quad positive curvature}, $\kappa(\omega)>0$. Thus, by continuity of $\gamma^{\prime}$ and $\kappa$ on $[0, \beta), A$ could be extended past $\omega$, contradicting the definition of $\omega$.
\end{proof}

Finally, before proving Theorem \ref{thm:uniqueness}, we need a last additional restriction on the behavior of 
$\gamma'$ when approaching $e_1$.
\begin{lemma}\label{lem:right-tangent}
If $F(0)>0$, then $\gamma_1(\beta)>0, \lim _{s \rightarrow \beta^{-}} \gamma^{\prime}(s)$ is in the fourth quadrant, and $\lim _{s \rightarrow \beta^{-}} \gamma^{\prime}(s) \neq(0,-1)$.    
\end{lemma}
\begin{proof}
It follows from Lemma \ref{lem:after lower curve} that $\gamma_1(\beta)>0$, as $\gamma^{\prime}(s)$ lies strictly in the fourth quadrant for all $s \in(\eta, \beta)$, and $\gamma_1(\eta)>0$. As $\kappa>0$ and $\gamma^{\prime}$ is in the fourth quadrant on $(\eta, \beta)$, the angle $\theta(s)$ that $\gamma^{\prime}(s)$ makes with the $e_1$-axis, measured counterclockwise in radians, must be a strictly increasing function on $(\eta, \beta)$ that is bounded above by $2 \pi$. Therefore, $\lim _{s \rightarrow \beta^{-}} \theta(s)$ exists and is in $(\theta(\eta), 2 \pi]$. It follows that $\lim _{s \rightarrow \beta^{-}} \gamma^{\prime}(s)$ exists, lies in the fourth quadrant, and is not $(0,-1)$.    
\end{proof}

\subsection{Proof of Theorem \ref{thm:uniqueness} and Corollary \ref{cor:isop_ineq}}
Combining the results from the previous sections, we now have all the ingredients to prove our main result.
\begin{proof}[Proof of Theorem \ref{thm:uniqueness}]
By Theorems \ref{thm:existence}, \ref{thm:regularity} and Proposition \ref{prop:spherical_symm} we deduce that the minimizer $E$ is spherically symmetric set with generating curve $\gamma.$ Then if $E$ is not the ball centered at the origin, $F(0)>0$ and therefore by the right tangent Lemma \ref{lem:right-tangent}, there exists a point such that $\gamma_1(\beta)>0, \lim _{s \rightarrow \beta^{-}} \gamma^{\prime}(s)$ is in the fourth quadrant, and $\lim _{s \rightarrow \beta^{-}} \gamma^{\prime}(s) \neq(0,-1)$. Then this implies that there exists $\varepsilon>0$ such that $\gamma(s) \cdot \dot{\gamma}(s) > 0$ on $s\in (\beta-\varepsilon, \beta)$ thereby contradicting the tangent restriction Lemma \ref{lem:tangent restrict}. Thus the curve $\gamma$ must be the circumference of a circle centered at the origin. Uniqueness follows by observing that up to negligible sets the only domains preserved by spherical symmetrization are the centered balls.
\end{proof}
As a consequence, an isoperimetric-type inequality follows.
\begin{proof}[Proof of Corollary \ref{cor:isop_ineq}]From the definition of $\Phi(r)=\abs{B_r}_f$ and $\tilde F(r)=\E(B_r)$ we have explicitly that
\[
\Phi(r)=n\omega_n\int_0^r t^{n-1}e^{\psi(t)}\,dt,
\]
and
\[
\tilde F(r)=n\omega_nr^{n-1}e^{\psi(r)} +n\omega_n\int_0^r t^{n-1}g(t) e^{\psi(t)}\,dt.
\]
In particular, $\Phi$ is monotone increasing and hence admits a well-defined inverse, proving that $\Efrak:=\tilde F\circ\Phi^{-1}$ is well defined. By Theorem \ref{thm:uniqueness}, for any $F\subset \R^n$ with volume $\abs{F}_f=v$, one has that
\[
\E(F)\geq\E(B_{\Psi(v)}),
\]
with equality if and only if $F$ is a centered sphere up to a negligible set. Equivalently
\[
\E(F)\geq\E(B_{\Phi^{-1}(v)})=\tilde F(\Phi^{-1}(v))=\Efrak{E}(\abs{F}_f),
\]
showing \eqref{eq:isoperimetric_ineq}. We are left to prove \eqref{eq:isop_profile}, which follows from the elementary computation
\begin{align*}
\Efrak'(v)&=\frac{\tilde F'}{\Phi'}\circ\Phi^{-1}(v)=\frac{n\omega_n r^{n-1}e^{\psi(r)}((n-1)r^{-1}+\psi'(r)+g(r))}{n\omega_n r^{n-1}e^{\psi(r)}}\Big\vert_{r=\Phi^{-1}(v)}\\
&=\Bigl(g(r)+\psi'(r)+\frac{n-1}{r}\Bigr)\Bigl\vert_{r=\Phi^{-1}(v)},
\end{align*}
and
\begin{align*}
\Efrak''(v)& =\Bigl(g(r)+\psi'(r)+\frac{n-1}{r}\Bigr)'\frac{1}{\Phi'(r)}\Bigl\vert_{r=\Phi^{-1}(v)}=\frac{r^2(g'(r)+\psi''(r))-(n-1)}{n\omega_nr^{n+1}e^{\psi(r)}}\vert_{r=\Phi^{-1}(v)},
\end{align*}
as wished.
\end{proof}

\section{Stability}\label{sec:stability}
\subsection{Preliminary results}
In this section, we establish a Fuglede-type result for nearly spherical sets under the assumption that $\psi,g$ are strictly admissible weights in the sense of Definition \ref{def:admissible_weights}. We first define nearly spherical sets from \cite{fusco-manna}:
\begin{definition}\label{defn:nearly-spherical}
Let $n \geq 2$. We say that a set $E$ is nearly spherical if there exist $R>0$ and a Lipschitz function $u: S^{n-1} \rightarrow(-1,1)$ such that
\begin{align}
\partial E=\left\{y: y=Rx(1+u(x)), x \in S^{n-1}\right\}. 
\end{align}
\end{definition}
With this in hand, we aim to show the following:
\begin{proposition}\label{prop:fuglede}
Given $\psi\in C^3(\R)$, $g\in C^2(\R)$ positive, and $R>0$ satisfying
\begin{equation}\label{eq:cond_on_R_1}
   \psi''(R)+g'(R)>0,\qquad \psi'(R)+g(R)\geq 0, 
\end{equation}
and
\begin{equation}\label{eq:cond_on_R_2}
    \psi'(R)>-\frac{(n-2)(n-1)}{R^2(\psi'(R)+g(R))+(n-1)R},
\end{equation}
there exist
\[
c=c(R,\psi,g,n)>0,\qquad\varepsilon=\varepsilon(R,\psi,g,n)>0,
\]
such that if $E$ is a nearly spherical set with $\|u\|_{W^{1, \infty}\left(S^{n-1}\right)}<\varepsilon$, $\left|B_R\right|_f=|E|_f$, then
\begin{equation}\label{eqn:nearly-spherical-stability-I}
\E(E)-\E(B_R) \geq cR^{n-1}e^{\psi(R)}\norm{u}_{L^2(S^{n-1})}^2.
\end{equation}
\end{proposition}

\begin{remark}
    Notice that conditions \eqref{eq:cond_on_R_1} and \eqref{eq:cond_on_R_2} are satisfied in particular if $\psi,g$ are strictly admissible weights in the sense of Definition \eqref{def:admissible_weights}, and $\psi$ is monotone increasing, or $\psi$, $g$ are $\kappa$-uniformly admissible and $R$ is small enough. In fact, it will be evident in the proof of the proposition that Equation \eqref{eq:cond_on_R_2} can be replaced with the weaker condition
    \[
        (R^2\psi'+R(n-1))(\psi'+g)+R(n-1)\psi'+R^2(\psi''+g')+(n-1)(n-2)>0,
    \]
    which we avoided to include in the statement for ease of reading. Finally, notice that Equation \eqref{eqn:nearly-spherical-stability-I} implies in particular Theorem \ref{thm:stability} in the class of nearly spherical sets since
    \begin{equation}
        \E(E)-\E(B_R)\geq cR^{n-1}e^{\psi(R)}\norm{u}_{L^2(S^{n-1})}^2\geq \tilde cR^{n-1}e^{\psi(R)}\abs{E\triangle B_R}^2_f,
    \end{equation}
    for some suitable constant $\tilde c>0$.
\end{remark}
\begin{proof}
We have the following expressions of the weighted perimeter and volume in terms of the graphical function $u$
\begin{align}
P_f(E) &=R^{n-1} \int_{S^{n-1}}(1+u(x))^{n-1} \sqrt{1+\frac{\left|\nabla_\tau u\right|^2}{(1+u)^2}} e^{\psi(R(1+u(x)))}\,d \Haus^{n-1} \\
|E|_f &= R^n \int_{S^{n-1}}(1+u(x))^n \int_0^1 t^{n-1} e^{\psi(R t(1+u(x))} \,d t\, d x\\
\G(E) &=  R^n \int_{S^{n-1}}(1+u(x))^n \int_0^1 g(t) t^{n-1} e^{\psi(R t(1+u(x))}\, d t\, d x,
\end{align}
where $\nabla_\tau u$ stands for the tangential gradient of $u$. Let $h(R)$ be any $C^3$ function. From the Taylor expansion
\begin{align*}
    &\frac{(1+u(x))^ne^{h(Rt(1+u(x)))}-e^{h(Rt)}}{e^{h(Rt)}}=(n+Rth'(Rt))u(x)\\
    &\quad+\frac12\Bigl(n(n-1)+2nRth'(R)+(Rt)^2(h''(R)+(h'(R))^2)\Bigr)u(x)^2+o_h(u(x)^2),
\end{align*}
and the smallness of $u$ in norm, we can estimate
\begin{align*}
    \frac{1}{R^n}&\Bigl(\int_E e^h\,dx-\int_{B_R}e^h\,dx\Bigr)=\int_{S^{n-1}}\int_0^1 t^{n-1}\Bigl((1+u)^ne^{h(Rt(1+u))}-e^{h(Rt)}\Bigr)\,dt\,d\Haus^{n-1}\\
    &= A\int_{S^{n-1}}u\,d\Haus^{n-1}+\frac {B}2\int_{S^{n-1}}u^2\,d\Haus^{n-1}+o(\norm{u}^2_{L^\infty(S^{n-1})}),
\end{align*}
where the coefficients $A$, $B$ can be computed by integration by parts as follows:
\[
A=\int_0^1 nt^{n-1}e^{h(Rt)}+Rth'(Rt)e^{h(Rt)}\,dt=e^{h(R)},
\]
and
\begin{align*}
B&=\int_0^1t^{n-1}e^{h(Rt)}\Bigl(n(n-1)+2nRth'(Rt)+(Rt)^2(h''(Rt)+(h'(Rt))^2)\Bigr)\,dt\\
&=e^{h(R)}((n-1)+Rh'(R)).
\end{align*}
In particular, setting $h=\psi$, it follows from the volume preserving condition on $u$, that there exists $C>0$ such that
\begin{equation}\label{eq:expansion_V}
    \int_{S^{n-1}}u\,d\Haus^{n-1}\geq-\frac{1}{2}(n-1+R\psi'(R)+C\varepsilon)\int_{S^{n-1}}u^2\,d\Haus^{n-1},
\end{equation}
and setting $h=\psi+\ln(g)$, we obtain the following lower bound
\begin{equation}\label{eq:expansion_G}
\begin{split}
    \frac{\mathcal G_f(E)-\mathcal G_f(B_R)}{e^{\psi(R)}R^n}&\geq g(R)\int_{S^{n-1}}u\,d\Haus^{n-1}\\
    &\quad+\frac{1}{2}g(R)\Bigl(n-1+R\Bigl(\psi'(R)+\frac{g'(R)}{g(R)}\Bigr)-C\varepsilon\Bigr)\int_{S^{n-1}}u^2\,d\Haus^{n-1}.
\end{split}
\end{equation}
A similar computation taking advantage of the elementary inequality 
\[
\sqrt{1+\frac{\abs{\nabla_\tau u}^2}{(1+u)^2}}-1\geq \frac{\abs{\nabla_\tau u}^2}{2(1+u)^2}\Bigl(1-\frac{\abs{\nabla_\tau u}}{4(1+u)}\Bigr)\geq\Bigl(\frac12-C\varepsilon\Bigr)\abs{\nabla_\tau u}^2,
\]
combined with a second-order Taylor expansion of $(1+u)^{n-1}e^{\psi(R(1+u))}$ around $u$ shows that the perimeter gap can be controlled from below by the expression
\begin{equation}\label{eq:expansion_P}
    \begin{split}
        &\frac{P_f(E)-P_f(B_R)}{e^{\psi(R)}R^{n-1}}\geq \Bigl(\frac12-C\varepsilon\Bigr)\int_{S^{n-1}}\abs{\nabla_\tau u}^2\,d\Haus^{n-1}+(n-1+R\psi'(R))\int_{S^{n-1}}u\,d\Haus^{n-1}\\
        &\quad+\frac12\Bigl((n-1)(n-2)+2(n-1)R\psi'(R)+R^2(\psi''(R)+\psi'(R)^2)-C\varepsilon\Bigr)\int_{S^{n-1}}u^2\,d\Haus^{n-1}.
    \end{split}
\end{equation}
Combining Equations \eqref{eq:expansion_G} and \eqref{eq:expansion_P}, we get that
\begin{align*}
    &\frac{P_f(E)-P_f(B_R)+\mathcal  G_f(E)-\mathcal G_f(B_R)}{e^{\psi(R)}R^{n-1}}\geq \Bigl(\frac12-C\varepsilon\Bigr)\int_{S^{n-1}}\abs{\nabla_\tau u}^2\,d\Haus^{n-1}\\
    &\quad+(n-1+R(\psi'(R)+g(R)))\int_{S^{n-1}}u\,d\Haus^{n-1}\\
    &\quad+\frac12\Bigl((n-1)(n-2)+2(n-1)R\psi'(R)+R^2(\psi''(R)+\psi'(R)^2)\Bigr)\int_{S^{n-1}}u^2\,d\Haus^{n-1}\\
    &\quad+\frac12\Bigl(Rg(R)(n-1)+R^2(\psi'(R)g(R)+g'(R))-C\varepsilon\Bigr)\int_{S^{n-1}}u^2\,d\Haus^{n-1}.
\end{align*}
If $\int_{S^{n-1}}u\,d\Haus^{n-1}\geq 0$ one can get the desired inequality since Equation \ref{eq:cond_on_R_2} implies that
\[
\Lambda:=R^2\psi'g+(n-1)R\psi'+(n-2)(n-1)+R(n-1)(\psi'+g)+R^2(\psi''+g')+R^2(\psi')^2>0,
\]
and hence by letting $\varepsilon<\frac{1+\Lambda}{4C}$ we get that
\begin{align*}
    \frac{\E(E)-\E(B_R)}{e^{\psi(R)}R^{n-1}}\geq \frac34\norm{\nabla_\tau u}_{L^2(S^{n-1})}^2+\frac34\Lambda\norm{u}_{L^2(S^{n-1})}^2. 
\end{align*}
Otherwise, taking advantage of Equation \eqref{eq:expansion_V} to control the oscillations of $u$, and the fact that our assumption ensures $\psi'+g>0$, we obtain the following bound
\begin{align*}
    \frac{\E(E)-\E(B_R)}{e^{\psi(R)}R^{n-1}}&\geq \Bigl(\frac12-C\varepsilon\Bigr)\norm{\nabla_\tau u}^2_{L^2(S^{n-1})}+\Bigl(\frac{D}{2}-C\varepsilon\Bigr)\norm{u}^2_{L^2(S^{n-1})},
\end{align*}
where $D$ is equal to
\begin{align*}
    D&=-(n-1+R(\psi'(R)+g(R)))(n-1+R\psi'(R))\\
    &\quad+(n-1)(n-2)+2(n-1)R\psi'(R)+R^2(\psi''(R)+\psi'(R)^2)\\
    &\quad +rg(R)(n-1)+R^2(\psi'(R)g(R)+g'(R))\\
    &=1-n+R^2(\psi''(R)+g'(R)).
\end{align*}
We conclude by developing $u$ over spherical harmonics. Letting $\{Y_{i,j}\}_{i\leq 0,j\leq Q{(i,n)}}$ be the orthonormal basis of $L^2(S^{n-1})$ induced by the spherical Laplacian, that is 
\[
-\Delta_\tau Y_{i,j}=i(i+n-2)Y_{i,j},
\]
and $a_{i,j}\in\R$ the projections of $u$ on $Y_{i,j}$ we get, subtracting from $u$ its mean, the following Poincaré-type estimate
\begin{align*}
\Bigl(\frac12-C\varepsilon\Bigr)&\norm{\nabla_\tau u}^2_{L^2(S^{n-1})}+\Bigl(\frac{D}{2}-C\varepsilon\Bigr)\norm{u-a_{0,1}}^2_{L^2(S^{n-1})}\\
&=\frac12 \sum_{i=1}^{+\infty}\sum_{j=1}^{Q(i,n)}\Bigl(i(i+n-2)(1-2C\varepsilon)+D-2C\varepsilon\Bigr)a_{i,j}^2\\
&\geq \frac12 \sum_{i=1}^{+\infty}\sum_{j=1}^{Q(i,n)}\Bigl(R^2(\psi''(R)+g'(R))-2nC\varepsilon\Bigr)a_{i,j}^2.
\end{align*}
We are left to control the term $a_{0,1}^2$, which turns out to be of order $O(\varepsilon^2)$: from the volume preserving properties of $u$ implying Equation \eqref{eq:expansion_V}, and the fact that we are treating the case in which $u$ has a negative mean, we get that
\begin{align*}
a_{0,1}^2&=\Bigl(\int_{S^{n-1}} u\,d\Haus^{n-1}\Bigr)^2\leq \frac{1}{4}(n-1+R\psi'(R)+C\varepsilon)^2\Bigl(\int_{S^{n-1}}u^2\,d\Haus^{n-1}\Bigr)^2\\
&\leq \frac{1}{4}(n-1+R\psi'(R)+C\varepsilon)^2n\omega_n\varepsilon^2\norm{u}_{L^2(S^{n-1})}^2\leq C\varepsilon^2\norm{u}_{L^2(S^{n-1})}^2.
\end{align*}
Hence, by collecting all the inequalities we get the estimate
\begin{align*}
    \frac{\E(E)-\E(B_R)}{e^{\psi(R)}R^{n-1}}\geq\frac12\Bigl(R^2(\psi''(R)+g'(R))-2nC\varepsilon+o(\varepsilon)\Bigr)\norm{u}_{L^2(S^{n-1})}^2
\end{align*}
obtaining the wished result taking $\varepsilon$ small enough, since by assumption we imposed $\psi''(R)+g'(R)>0$.
\end{proof}
\subsection{General Stability}
In this section, we will prove Theorem \ref{thm:stability} following the arguments in~\cite{fusco-manna}. The main difference here is that we also have to keep track of the potential energy which makes some arguments non-trivial. We first introduce the key concept of almost-minimizing sets and with that the well-known De-Giorgi $\varepsilon$-regularity Theorem, which will be necessary to bridge the following argument with the Fuglede-type stability of Proposition \ref{prop:fuglede}, in order to obtain the desired conclusion.
\begin{definition}\label{defn:omega-minimizer}
Let $E \subset \mathbb{R}^n$ be a set of locally finite perimeter, $\omega \geq 0$, and let $\Omega$ be an open subset of $\mathbb{R}^n$. We say that $E$ is an $\omega$-minimizer of the perimeter in $\Omega$ if for every ball $B_\rho(x) \subset \subset \Omega$ with $\rho<1$ and for any set $F$ of locally finite perimeter such that $E \triangle F \subset \subset B_\rho(x)$ it holds
$$
P\left(E ; B_\rho(x)\right) \leq P\left(F ; B_\rho(x)\right)+\omega \rho^n.
$$
\end{definition}
\begin{definition}\label{defn:convergence-in-measure}
Let $E_h$ and $E$ be measurable sets of $\mathbb{R}^n$ and $\Omega$ is an open set. Then one says that $E_h \rightarrow E$ in measure in $\Omega$ if $\left|E_h \triangle E \cap \Omega\right| \rightarrow 0$.
\end{definition}
A consequence of De-Giorgi's $\varepsilon$-regularity theorem, see Theorem 2.2 in \cite{fusco-manna} states that a sequence $\omega$-minimizers sets converge in a graphical sense. 
\begin{theorem}\label{thm:convergence-implies-graphicality}
Assume that $E_h, E$ are equibounded $\omega$-minimizers of the perimeter in $\mathbb{R}^n$ such that $E_h \rightarrow E$ in measure. If $E$ is of class $C^2$ then for $h$ large $E_h$ is of class $C^{1, \frac{1}{2}}$ and there exists a function $v_h: \partial E \rightarrow \mathbb{R}$ such that
$$
\partial E_h=\left\{x+v_h(x) \nu_E(x), x \in \partial E\right\}
$$
Moreover, $\left\|v_h\right\|_{C^{1, \alpha}(\partial E)} \rightarrow 0$ for $0<\alpha<\frac{1}{2}$
\end{theorem}
We are now ready to start with some preliminary lemmas.
\begin{lemma}\label{lem:continuity-perimeter}
Let $r>0$ such that $\psi(r)>\psi(0)$, $g\geq 0.$ Let $\varepsilon>0$, then there exists $\delta>0$ such that for every set of finite perimeter $E$ with $|E|_f=\left|B_r\right|_f$, if $\E(E)-\E\left(B_r\right)<\delta$ then $\left|E \Delta B_r\right|_f<\varepsilon$.
\end{lemma}
\begin{proof}
We proceed by contradiction. Assume that there exists some \(\varepsilon_0 > 0\) such that for every \(k \in \mathbb{N}\), we can find a set \(E_k\) satisfying the following conditions:
\[
\left|E_k\right|_f = \left|B_r\right|_f = v, \quad \E(E_k) - \E(B_r) \leq \frac{1}{k}, \text{ but } \left|E_k \Delta B_r\right|_f \geq \varepsilon_0.
\]
Since \(\psi(r) > \psi(0)\) and \(g \geq 0\), for sufficiently large \(k\), we have
\[
e^{\psi(0)} \Cp(E) \leq \Pe(E_k) \leq \E(E_k) \leq 2 \E(B_r).
\]
Thus, \(\{E_k\}_{k \in \mathbb{N}}\) is a family of equibounded sets, so there exists a subsequence (which we still denote by \(\{E_k\}\)) and a set \(E\) such that \(\chi_{E_k} \to \chi_E\) in \(L^1_{\text{loc}}(\mathbb{R}^n)\), with
\[
\E(E) \leq \liminf_{k \to \infty} \E(E_k) \leq \E(B_r).
\]
Then following the same argument as in the proof of Lemma 4.2 in ~\cite{fusco-manna} we deduce that $|E|_f=v$. Then by Theorem~\ref{thm:uniqueness}, we have \(E = B_r\). This contradicts the assumption that \(\left|E_k \Delta B_r\right|_f \geq \varepsilon_0\).

\end{proof}

\begin{lemma}\label{lem:calibration}
Let $r>0$ such that $\psi(r)>\psi(0)$ and either
\begin{align} 
\Lambda_1 \geq 0, \Lambda_2 \geq \frac{8(n+1)}{r}+2(g(r)+\psi^{\prime}(2 r)),\quad \text{or}\quad \Lambda_1 \geq n-1+r \psi^{\prime}(r)+rg(r), \Lambda_2=0. 
\end{align}
Then $B_r$ is the only minimizer of the functional defined for a measurable set $E \subset \mathbb{R}^n$ as
\begin{align}
\mathcal{J}_{\Lambda_1, \Lambda_2}(E)=\E(E)+\Lambda_1 \left|| E|_f-|B_r|_f\right|+\Lambda_2 |E \Delta B_r|_f    
\end{align}
\end{lemma}
\begin{proof}
Let $\phi \in C^\infty_c([0,1])$ such that $\phi(t)\equiv 1$ for $t \in\left[\frac{r}{2}, \frac{3 r}{2}\right]$, $\phi(t)=0$ outside of the interval $\left[\frac{r}{4}, \frac{7 r}{4}\right]$ and $\left\|\phi^{\prime}(t)\right\|_{L^{\infty}} \leq 8 / r$. Consider the smooth vector field $X(x)=\phi(|x|) \frac{x}{|x|}$. Then $\|X\|_{L^{\infty}}\leq 1$ and $\|\operatorname{div} X\|_{L^\infty}\leq \frac{4n+4}{r}.$ Then integration by parts implies
\begin{align}
\int_{\partial^* E} e^{\psi(|x|)} \ d \mathcal{H}^{n-1} & \geq \int_{\partial^* E} e^{\psi(|x|)}\left\langle X, \nu_E\right\rangle \ d \mathcal{H}^{n-1} \\
& =\int_E\left(\operatorname{div} X+\psi^{\prime}(|x|) \frac{\langle X, x\rangle}{|x|}\right) e^{\psi(|x|)} \ d x    
\end{align}
while
\begin{align}
\int_{\partial B_r} e^{\psi(|x|)} \ d \mathcal{H}^{n-1}=\int_{B_r}\left(\operatorname{div} X+\psi^{\prime}(|x|) \frac{\langle X, x\rangle}{|x|}\right) e^{\psi(|x|)} \ d x.  
\end{align}
On the other hand since $g\geq 0$ and is monotone increasing we have
\begin{align}
    \int_{E} g e^{\psi} \ dx - \int_{B_r} g e^{\psi} \ dx \geq  -g(r)|E\Delta B_r|_f.
\end{align}
Thus we get
\begin{align}
\mathcal{J}_{\Lambda_1, \Lambda_2}(E)-\mathcal{J}_{\Lambda_1, \Lambda_2}\left(B_r\right) & \geq\left(\Lambda_2- g(r)-\|\operatorname{div} X\|_{L^{\infty}\left(\mathbb{R}^n\right)}-\left\|\psi' X\right\|_{L^{\infty}\left(\mathbb{R}^n\right)}\right)\left|E \Delta B_r\right|_f \\
& \geq\left(\Lambda_2-\frac{4 n+4}{r}-g(r) - \psi^{\prime}(2 r)\right)\left|E \Delta B_r\right|_f.  
\end{align}
Thus in the first case, we see that the lower bound on $\Lambda_2$ gives the desired inequality. On the other hand, if $\Lambda_2=0$ then Theorem \ref{thm:uniqueness} implies that $\mathcal{J}_{\Lambda_1, \Lambda_2}$ is minimized by balls centered at the origin. On such balls, the value of the functional is given by
\begin{align}
\mathcal{J}_{\Lambda_1, \Lambda_2}\left(B_{\varrho}\right)=n \omega_n \varrho^{n-1} e^{\psi(\varrho)}+ n\omega_n \int_0^\rho g(t)e^{\psi(t)}t^{n-1} \ dt + n\omega_n\Lambda_1\left|\int_{\varrho}^r e^{\psi(t)} t^{n-1} \ d t\right|:=f(\rho).
\end{align}
Then $\Lambda_1 \geq (n-1)+r\psi'(r)+rg(r)$ implies that $f(\rho)$ achieves unique minima at $\rho=r.$
\end{proof}
We recall in the following Definition the function $\Phi$ appearing in Corollary \ref{cor:isop_ineq}.
\begin{definition}\label{defn:psi}
Given a ball of radius $r>0$, we define $\Psi(t)$ to be equal to the radius $r$ such that $|B_r|_f=t.$ More precisely, denote $\Phi, \Psi: \mathbb{R}_{+} \rightarrow \mathbb{R}_{+}$ as
\begin{align}
\Phi(s)=n \omega_n \int_0^s t^{n-1} e^{\psi(t)} d t, \quad \Psi(t)=\Phi^{-1}(t).
\end{align}
 where $s, t \geq 0$.  
\end{definition}
\begin{remark}
Note that $\Psi$ is well defined since $\Phi$ is a strictly increasing function and furthermore from Lemma 5.1 in~\cite{fusco-manna} $\Psi$ has the the following properties
\begin{enumerate}
    \item $\Psi \in C^{\infty}(0, \infty)$
    \item For $t>0$ 
    \begin{align}
    \Psi^{\prime}(t) =\frac{1}{n \omega_n \Psi^{n-1}(t) e^{\psi(\Psi(t))}},\quad t \leq n \omega_n \Psi(t)^n e^{\psi(\Psi(t))} \label{eqn:second-equation}
    \end{align}
\end{enumerate}
\end{remark}
\begin{lemma}\label{lem:perimeter}
Let $E \subset \mathbb{R}^n$ be a set of finite perimeter such that $\left|E \setminus B_r\right|_f \leq \eta<1$. There exists $R_E \in[r, r+4 \Psi(\eta)]$ such that
\begin{align}
\E(E) \leq \E\left(E \cap B_{R_E}\right)-\frac{\left|E \setminus B_{R_E}\right|_f}{2 \Psi(\eta)}.
\end{align}
\end{lemma} 
\begin{proof}
Arguing by contradiction, $r \leq t \leq r+4 \Psi(\eta)$ it holds
\begin{align}\label{eqn:contra-hypoth}
\E\left(E \cap B_t\right)>\E(E)-\frac{\left|E \setminus B_t\right|_f}{2\Psi(\eta)}.    
\end{align}
Set $v(t)=\left|E \setminus B_t\right|_f$. Then for a.e. $t>0$
\begin{align}
v^{\prime}(t)=-e^{\psi(t)} \mathcal{H}^{n-1}\left(E \cap \partial B_t\right).
\end{align}
Since
\begin{align}
 \E(E) &= \Pe(E) + \G(E)  \\
 &\geq \Pe\left(E \cap B_t\right)+\Pe\left(E \setminus B_t\right)+2 v^{\prime}(t)  + \G(E\cap B_t) + \G(E\setminus B_t)
\end{align}
then \eqref{eqn:contra-hypoth} implies
\begin{align}
2 v^{\prime}(t)+\E(E\setminus B_t)<\frac{v(t)}{2 \Psi(\eta)}
\end{align}
The weighted isoperimetric inequality hence gives
\begin{align}
2 v^{\prime}(t)+n \omega_n \Psi(v(t))^{n-1} e^{\psi(\Psi(v(t)))} + n\omega_n \int_{0}^{\Psi(v(t))} g(s) e^{\psi(s)} s^{n-1}\ ds<\frac{v(t)}{2 \Psi(\eta)}.  
\end{align}
Then using \eqref{eqn:second-equation}
\begin{align}
\frac{v(t)}{\Psi(\eta)} \leq n w_n \Psi(v(t))^{n-1} e^{\psi(\Psi(v(t)))}    
\end{align}
and the fact that $g\geq 0$ we get
\begin{align}
\Psi(v(t))^{n-1} e^{u \Psi(v(t))}<-\frac{4}{n \omega_n} v^{\prime}(t) \quad \text { for all } t \in[r, r+4 \Psi(\eta)]. 
\end{align}
Then integrating the above inequality yields
\begin{align}
4 \Psi(\eta) & <-\frac{4}{n \omega_n} \int_r^{r+4 \Psi(\eta)} \frac{v^{\prime}(t)}{\Psi(v(t))^{n-1} e^{\Psi(v(t))}} d t \\
& =\frac{4}{n \omega_n} \int_{v(r+4 \Psi(\eta))}^{v(r)} \frac{1}{\Psi(s)^{n-1} e^{\Psi(s)}} d s \\
& =4(\Psi(v(r))-\Psi(v(r+4 \Psi(\eta)))    
\end{align}
which is not possible.
\end{proof}
\begin{lemma}\label{lem:fixing-R0}
Let $r>0$ such that $\psi(r)>\psi(0), \Lambda_1 \geq n-1+r (g(r)+\psi^{\prime}(r))$ and $\Lambda_2>0$. There exist $0<\alpha_1<\frac{\Lambda_2}{2 \Lambda_2+1}$ such that for any $\alpha \in\left[0, \alpha_1\right]$ the functional
\begin{align}
\mathcal{J}_{\Lambda_1, \Lambda_2, \alpha}(E)=\E(E)+\Lambda_1 \left| |E|_f- |B_r|_f\right|+\Lambda_2 \left| |E\Delta B_r|_f -\alpha \right|,\quad E \subset \mathbb{R}^n.    
\end{align}
admits a bounded minimizer $E \subset B_{R_0}$ where $R_0 = r+\Psi(1).$
\end{lemma}
\begin{proof}
Let $\{E_k\}_{k\in \mathbb{N}}$ be a minimizing sequence such that
$$
\mathcal{J}_{\Lambda_1, \Lambda_2, \alpha}\left(E_k\right) \leq \inf _{E \subset \mathbb{R}^n} \mathcal{J}_{\Lambda_1, \Lambda_2, \alpha}(E)+\frac{\alpha_1}{k} \leq \E\left(B_r\right)+\Lambda_2 \alpha+\frac{\alpha_1}{k}
$$
Combining this with Lemma \ref{lem:calibration} we get
$$
\E\left(B_r\right)+ \Lambda_2 \left| |E_k \Delta B_r|-\alpha\right| \leq \mathcal{J}_{\Lambda_1, \Lambda_2, \alpha}\left(E_k\right) \leq \E\left(B_r\right)+\Lambda_2 \alpha+\frac{\alpha_1}{k}.
$$
Thus,
$$
\left|E_k \setminus B_r\right|_f \leq\left|E_k \Delta B_r\right|_f \leq\left(2+\frac{1}{k \Lambda_2}\right) \alpha_1.
$$
Set $\beta:=\left(\frac{2 \Lambda_2+1}{\Lambda_2}\right) \alpha_1<1$. Thus Lemma \ref{lem:perimeter} implies that there exists $r_k \in[r, r+4 \Psi(\beta)]$ such that
$$
\E\left(E_k \cap B_{r_k}\right) \leq \E\left(E_k\right)-\frac{\left|E_k \setminus B_{r_k}\right|_f}{2 \Psi(\beta)}
$$
Thus if we denote $\Tilde{E}_{k}=E_k \cap B_{r_k}$ then
\begin{align}
\mathcal{J}_{\Lambda_1, \Lambda_2, \alpha}\left(\Tilde{E}_{k}\right) &\leq \E\left(E_k\right)-\frac{\left|E_k \setminus B_{r_k}\right|_f}{2 \Psi(\beta)}+\Lambda_1 \left| |\Tilde{E}_{k}|_f -|B_r|_f \right|+\Lambda_2 \left| |\Tilde{E}_{k}\Delta E_k|_f-\alpha\right|  \\
& \leq \mathcal{J}_{\Lambda_1, \Lambda_2, \alpha}\left(E_k\right)+\left(\Lambda_1-\frac{1}{2 \Psi(\beta)}\right)\left|E_k \setminus B_{r_k}\right|_f+\Lambda_2\left|\Tilde{E}_{k} \Delta E_k\right|_f \\
& =\mathcal{J}_{\Lambda_1, \Lambda_2, \alpha}\left(E_k\right)+\left(\Lambda_1+\Lambda_2-\frac{1}{2 \Psi(\beta)}\right)\left|E_k \setminus B_{r_k}\right|_f.    
\end{align}
Thus for $\beta$ small enough, $\Tilde{E}_{k}$ is a minimizing sequence such that $\Tilde{E}_{k} \subset B_{R_0}.$ The minimizer is then obtained by a standard compactness argument.
\end{proof}
\begin{lemma}\label{lem:omega-minimizer}
Given $\Lambda_1, \Lambda_2 \geq 0$, there exists $\omega \geq 0$ such that if $E \subset B_{R_0}$ is a minimizer of $\mathcal{J}_{\Lambda_1, \Lambda_2, \alpha}$ with $\alpha \geq 0$, then $E$ is an $\omega$-minimizer of the perimeter in $B_{2 R_0}$.
\end{lemma}
\begin{proof}
See proof of Lemma 5.4 in \cite{fusco-manna}. Since the definition of an $\omega$-minimizers provides local control on a ball $B_{\rho}(x)$ on the unweighted perimeter of the set $E$, the key point is to first get rid of the weights by using trivial bounds by its minimum and maximum on $B_{\rho}(x)$ and then use the regularity of the weight to bound the oscillation on $B_{\rho}(x)$ to get the desired bound. 

\end{proof}
\begin{lemma}\label{lem:almost-spherical}
Let $r>0$ such that $\psi(r)>\psi(0)$, let $\Lambda_1, \Lambda_2$ satisfy the assumptions of Lemma \ref{lem:calibration} and let $\varepsilon_k \rightarrow 0$. Let $F_k$ be a sequence of equibounded minimizers of $\mathcal{J}_{\Lambda_1, \Lambda_2, \varepsilon_k}$. Then up to a subsequence $F_k \rightarrow B_r$ in $C^{1, \alpha}$ for all $\alpha<\frac{1}{2}$, i.e. for each $i\in \mathbb{N}$, there exists a graph $u_k \in C^{1, \frac{1}{2}}\left(\mathbb{S}^{n-1}\right)$ such that
$$
\partial F_k=\left\{r x\left(1+u_k(x)\right), x \in \mathbb{S}^{n-1}\right\} \quad \text { with } \quad\left\|u_k\right\|_{C^{1, \alpha}\left(\mathbb{S}^{n-1}\right)} \rightarrow 0 \quad \text { for } \alpha\in (0,1/2) \text {. }
$$
\end{lemma}
\begin{proof}
Since $F_k$ minimizes the functional $\mathcal{J}_{\Lambda_1, \Lambda_2, \varepsilon_k}$, we obtain the following inequality:
\[
e^{\psi(0)} P(F_k) \leq \Pe(F_k) \leq \mathcal{J}_{\Lambda_1, \Lambda_2, \varepsilon_k}(F_k) \leq \E(B_r) + \varepsilon_k.
\]
As $\varepsilon_k \to 0$ and the sets $\{F_k\}$ are uniformly bounded, there exists a set $F$ with finite perimeter such that, possibly after passing to a subsequence, we have $|F_k \Delta F| \to 0$. Furthermore, for any set $E$ with finite weighted perimeter and for each $k \in \mathbb{N}$, we have
\[
\mathcal{J}_{\Lambda_1, \Lambda_2, \varepsilon_k}(F_k) \leq \mathcal{J}_{\Lambda_1, \Lambda_2, \varepsilon_k}(E).
\]
Letting $k \to \infty$, this gives $\mathcal{J}_{\Lambda_1, \Lambda_2}(F) \leq \mathcal{J}_{\Lambda_1, \Lambda_2}(E)$, which shows that $F$ minimizes $\mathcal{J}_{\Lambda_1, \Lambda_2}$. From Lemma \ref{lem:calibration}, we conclude that $F_k \to B_r$ in measure. Finally, the result follows from Lemma \ref{lem:omega-minimizer} and Theorem \ref{thm:convergence-implies-graphicality}.
\end{proof}
We are now ready to prove the main theorem of this section.
\begin{proof}[Proof of Theorem \ref{thm:stability}]
We split the proof into two cases, first when the asymmetric difference is small $|B_r \Delta E|_f < \delta$ and the other when it is large $|B_r \Delta E|_f \geq \delta$. In the latter case, Lemma \ref{lem:continuity-perimeter} implies the existence of $\sigma>0$ such that if $\left|E \Delta B_r\right|_f \geq \delta$ then $\E(E)-\E\left(B_r\right) \geq \sigma$ and thus
$$
\E(E)-\E\left(B_r\right) \geq \frac{\sigma}{4\left|B_r\right|_f^2}\left|E \Delta B_r\right|_f^2.
$$
Therefore, it remains to deal with the case when $|B_r \Delta E|_f < \delta.$ To this end we aim to show that for all $r>0$ such that $\psi^{\prime \prime}(r)>0$ there exists $\delta>0$ such that if $\left|B_r \Delta E\right|_f<\delta$ and $|E|_f=\left|B_r\right|_f$ then
\begin{align}\label{eqn:goal-stability}
\E(E)-\E\left(B_r\right) \geq C\left|B_r \Delta E\right|_f^2    
\end{align}
for some constant $C>0$. As usual, we will argue by contradiction. Thus let $\{E_k\}_{k\in \mathbb{N}}$ be a sequence of sets of finite perimeter satisfying $|E_k|_f = |B_r|_f$, $|E_k \Delta B_r|\to 0$ but
\begin{align}\label{eqn:contra-hyp}
\E\left(E_k\right) \leq \E\left(B_r\right)+C\left|E_k \Delta B_r\right|_f^2.    
\end{align}
The idea now is to use this family of counter-examples to construct a sequence of minimizers for $\mathcal{J}_{\Lambda_1,\Lambda_2,\varepsilon_k}$ that are equibounded and therefore by Lemma~\ref{lem:almost-spherical} are almost spherical sets. But for such sets, we have seen a sharp quantitative stability estimate in Proposition \ref{prop:fuglede} which will yield a contradiction to \eqref{eqn:contra-hypoth} for $C>0$ small enough. More precisely, set
\begin{align}
\varepsilon_k=\left|E_k \triangle B_r\right|_f,\quad \Lambda_1 \geq n-1+r \psi^{\prime}(r)+rg(r),\quad \Lambda_2 >0.
\end{align}
Then Lemma \ref{lem:fixing-R0} yields a bounded minimizer $F_k \subset B_{R_0}$ of the functional $\mathcal{J}_{\Lambda_1, \Lambda_2, \varepsilon_k}$, where $R_0=r+4 \Psi(1)$. Also recall that by Lemma \ref{lem:almost-spherical} we see that $F_k \rightarrow B_r$ (up to a subsequence possibly) in $C^{1, \alpha}$ for all $\alpha \in(0,1 / 2)$ and thus in particular for ball centered at the origin such that $|B_{r_k}|_f= |F_k|_f$ by Proposition~\ref{prop:fuglede} we have
\begin{align}\label{eqn:f_k-stability}
    c_0 |F_k \Delta B_{r_k}|_f^2 \leq \E(F_k)-\E(B_{r_k})
\end{align}
for some constant $c_0>0.$ To arrive at a contradiction we will aim to show that
\begin{align}\label{eqn:main-estimate}
    \E(F_k)-\E(B_{r_k}) &= \E(F_k)-\E(B_{r}) + \E(B_{r})-\E(B_{r_k})\\
    &\leq C\Tilde{C} |F_k \Delta B_{r_k}|_f^2.
\end{align}
Thus choosing $C$ such that $C\Tilde{C}/c_0<1$ gives us a contradiction to \eqref{eqn:f_k-stability}. To this end, we use the minimality of $F_k$ to get 
\begin{align}\label{eqn:J-comparison}
\mathcal{J}_{\Lambda_1, \Lambda_2, \varepsilon_k}\left(F_k\right) \leq \mathcal{J}_{\Lambda_1, \Lambda_2, \varepsilon_k}\left(E_k\right)=\E\left(E_k\right) \leq \E\left(B_r\right)+C \varepsilon_k^2.  
\end{align}
From this inequality, if $\Lambda_2$ is chosen large enough $\Lambda_2 \geq 2\left(\frac{8(n+1)}{r}+2(g(r)+\psi^{\prime}(2 r))\right)$, by applying Lemma \ref{lem:calibration} with $\Lambda_2$ replaced by $\Lambda_2 / 2$ and $\Lambda_1=0$, we have
\begin{align}
\E\left(F_k\right)+\left.\Lambda_2| | F_k \Delta B_r\right|_f-\varepsilon_k \mid & \leq \E\left(B_r\right)+C \varepsilon_k^2 \\
& \leq \E\left(F_k\right)+\frac{\Lambda_2}{2}\left|F_k \Delta B_r\right|_f+C \varepsilon_k^2.
\end{align}
Thus for $k\gg 1$ we have
\begin{align}\label{eqn:sym-diff-lb}
\left|F_k \Delta B_r\right|_f \geq \frac{\varepsilon_k}{2} .   
\end{align}
Assume now that $\Lambda_1 \geq 2\left(n-1+r \psi^{\prime}(r)+rg(r)\right)$. By \eqref{eqn:J-comparison} and Lemma \ref{lem:calibration} with $\Lambda_1$ replaced by $\Lambda_1 / 2$ and $\Lambda_2=0$ we have
\begin{align}
\E\left(F_k\right)+\Lambda_1 \left| |F_k|_f  - |B_r|_f \right| & \leq \E\left(B_r\right)+C \varepsilon_k^2 \\
& \leq \E\left(F_k\right)+\frac{\Lambda_1}{2} \left| |F_k|_f  - |B_r|_f \right|+C \varepsilon_k^2.
\end{align}
Thus
\begin{align}\label{eqn:volume-diff}
\left|\left|F_k\right|_f-\left|B_r\right|_f\right| \leq 2 C \varepsilon_k^2    
\end{align}
Recalling that $\left|B_{r_k}\right|_f=\left|F_k\right|_f$ we get
$$
\left|F_k \Delta B_r\right|_f \leq\left|F_k \Delta B_{r_k}\right|_f+\left|B_r \Delta B_{r_k}\right|_f \leq\left|F_k \Delta B_{r_k}\right|_f+2 C \varepsilon_k^2
$$
and thus when $k$ is large, using \eqref{eqn:sym-diff-lb}, we have $\left|F_k \Delta B_r\right|_f \leq 2\left|F_k \Delta B_{r_k}\right|_f.$ Combing with \eqref{eqn:volume-diff} and \eqref{eqn:sym-diff-lb} we get the desired \eqref{eqn:main-estimate}
$$
\E\left(B_r\right) - \E(B_{r_k}) \leq C\left|r-r_k\right| \leq C' C \varepsilon_k^2 \leq \tilde{C} C\left|F_k \Delta B_{r_k}\right|_{f}^2
$$
which yields a contradiction to Proposition~\ref{prop:fuglede} provided we choose a constant $C$ small enough.
\end{proof}

\printbibliography
\end{document}